\documentclass[reqno,11pt,a4paper]{amsart}
\usepackage[T1]{fontenc}
\usepackage{graphicx,mathtools,amssymb,amsthm,thmtools}
\usepackage{xcolor,babel,amsmath,enumitem,bm}
\usepackage{appendix}
\usepackage{tikz-cd,caption}
\usetikzlibrary{patterns, arrows.meta, positioning}
\usepackage[numbers]{natbib}
\newtheorem{theorem}{Theorem}[section]
\newtheorem{lemma}{Lemma}[section]
\newtheorem{definition}{Definition}[section]

\newtheorem{remark}{Remark}[section]
\numberwithin{equation}{section}
\usepackage[%
colorlinks=true,  % 启用彩色链接而非边框
linkcolor=blue,  % 内部链接颜色
citecolor=blue,  % 引用链接颜色
urlcolor=black,   % URL链接颜色
anchorcolor=black,% 锚点链接颜色
]{hyperref}
\makeatletter
\@ifundefined{theHtheorem}{}{%
}
\@ifundefined{theHlemma}{}{%
}
\@ifundefined{theHremark}{}{%
}
\makeatother
\usepackage{mathrsfs}
\allowdisplaybreaks[4]
\begin{document}
	\title[ commutator estimates and their applications]
	{commutator estimates and their applications to the transport-type equations }
	\author{Qianyuan Zhang}
	\address{Qianyuan Zhang\newline
		School of Mathematics and Statistics\\
		Huazhong University of Science and Technology, Wuhan 430074,  China}
	\email{qianyuanzhang@hust.edu.cn}
	
	\author{Kai Yan\textsuperscript{*}}
	\thanks{\noindent $^*$Corresponding author.}
	\address{Kai Yan (Corresponding author)\newline
		School of Mathematics and Statistics\\
		Huazhong University of Science and Technology, Wuhan 430074,  China}
	\email{kaiyan@hust.edu.cn}
	
	\begin{abstract}
	In this paper, we derive new commutator estimates in the Triebel-Lizorkin spaces by employing Bony's para-product decomposition, the Nikol'skij representation, and the Fefferman-Stein vector-valued maximal function. These estimates are then applied to develop a general theory for transport equations. Although analogous results are already available in the setting of Besov spaces, the methods developed there do not carry over directly to the Triebel-Lizorkin case. Our approach works for Triebel-Lizorkin spaces and, as a byproduct, also yields the corresponding results in Besov spaces. All proofs are presented in a unified manner that applies to both scales of function spaces, thereby extending and sharpening previous results on transport equations in these frameworks. Furthermore, the general theory we obtain is widely applicable to evolution equations, including incompressible and compressible ideal fluid flows, shallow water waves, and related models. As an illustration, we consider the two-component Euler-Poincaré system. Using the theoretical framework developed herein, we establish its local well-posedness and a blow-up criterion in both sub-critical and critical Triebel-Lizorkin spaces.
	\end{abstract}

	\maketitle
	
	
	\vskip 0.2cm
	
	
	\setcounter{equation}{0}

	\section{Introduction}
	\subsection{Commutator estimates}
	The primary objective of this paper is to establish commuta\-tor estimates in the Triebel-Lizorkin spaces. Commutator estimates are fundamen\-tal tools in analysis and the theory of nonlinear partial differential equations. In the study of nonlinear PDEs, particularly in fluid dynamics and dispersive equations, one frequently encounters nonlinear terms of the form $vf$ or more general composi\-tions. The classical Leibniz rule for derivatives, $\partial(vf)=(\partial v)f+v(\partial f)$, provides a precise description of how derivatives distribute over products. However, when dealing with fractional derivatives  $(-\Delta)^{s/2}$ or more general Fourier multipliers $T$, the situation becomes considerably more subtle. The commutator $[Tv,f]=T(vf)-vTf$ emerges as a fundamental quantity that measures the deviation from a simple Leibniz rule and enables the redistribution of derivatives in a way that facilitates a priori estimates. In \cite{MR951744}, Kato and Ponce proved the following commutator estimates:
	\[ \|J^s(vf)-vJ^sf\|_{L^p}\leq C\big(\|\nabla v\|_{L^\infty}\|J^{s-1}f\|_{L^p}+\|J^sv\|_{L^p}\|f\|_{L^\infty}\big), \]
	where $v,f\in\mathscr{S}(\mathbb{R}^d)$, $s>0,1<p<\infty$,   $J^s\triangleq(I-\Delta)^{\frac{s}{2}}$, and $C$ is a constant depending only on $d, p, s$. In recent decades, extensive work has been devoted to extending Kato-Ponce type estimates to broader classes of operators and function spaces. Danchin \cite{MR1827098} established commutator estimates for $[v,\Delta_j]\cdot\nabla f$ in Besov spaces, where $\Delta_j$ denotes the Littlewood-Paley projection operator (see \eqref{def:Delta_j} for details). Under suitable regularity assumptions, he derived norm inequalities linking the commutator to $\nabla v$ and $f$. Subsequently, these estimates were extended to the setting of Triebel-Lizorkin spaces, see \cite{MR2592288,MR4240785,2601.10071} for further discussions. Moreover, Takada \cite{MR2733256} investigated the optimal differential order for these estimates by means of counterexamples in Besov and Triebel-Lizorkin spaces. Despite these advances, the development of Kato-Ponce type commutator estimates in Triebel-Lizorkin spaces remains incomplete. 
	
	As is well known, the Triebel-Lizorkin spaces $F^s_{p,q}$ and their homogeneous versions $\dot{F}^s_{p,q}$ provide a unified framework encompassing many classical function spaces arising in the theory of partial differential equations. Notable identifications include (see the details in \cite{MR2768550,MR781540}):
	\begin{itemize}[left=0pt]
		\item The Sobolev spaces $W^{s,p}=F^s_{p,2}$ for $s\in\mathbb{R},1<p<\infty.$ In particular, $H^s=F^s_{2,2}$.
		\item The H\"{o}lder-Zygmund spaces $\mathscr{C}^s=F^s_{\infty,\infty}$ for $s>0. $
		\item The Hardy spaces $\mathcal{H}_p=\dot{F}^0_{p,2}$ for $0<p<\infty$. 
		\item The space of functions of bounded mean oscillation $BMO=\dot{F}^0_{\infty,2}$. 
		\item The Besov spaces: $B^s_{p,\min(p,q)} \hookrightarrow F^s_{p,q}\hookrightarrow B^s_{p,\max(p,q)}$ $(s\in\mathbb{R},(p,q)\in[1,\infty)\times [1,\infty]$ or $p=q=\infty$). In particular, $F^s_{p,p}=B^s_{p,p}$ $(s\in\mathbb{R},1\leq p\leq \infty)$.
	\end{itemize}
	This unified perspective serves as the primary motivation for refining the Kato-Ponce type commutator estimates in this context. The commutator estimates in Triebel-Lizorkin spaces, however, presents distinct and considerable challenges compared to that in Besov spaces. From the definition of Besov spaces (see Section \ref{section:2}), it becomes evident that the spaces $L^p_\Omega$ serve as the fundamental building blocks, where $L_\Omega^p\triangleq \{f\in L^p\colon\text{supp}\hat{f}\subset\Omega\}$ with $1\leq p\leq \infty$. Indeed, each dyadic piece $\Delta_j f$ of a Besov function belongs to $L_\Omega^p$ with $\Omega=\mathscr{B}(\frac{8}{3}2^{j+1})$ (see \eqref{def:Delta_j}), and these pieces can be treated separately within the theory of the spaces $L_\Omega^p$. Bernstein's lemma (Lemma \ref{lem:Bernstein}) then allows us to control the dependence of the constants on $j$ in the relevant estimates via homogeneity arguments, after which the estimates are summed up in the sense of \eqref{def:norm_nonhomogeneous B}. This reduction makes many problems in the Besov setting more convenient and tractable. The Triebel-Lizorkin spaces, by contrast, do not admit such a reduction. The pieces $\Delta_jf$ cannot be treated one after another, they must be considered simultaneously, as seen from the definition of the $F^s_{p,q}$-norm (see \eqref{def:norm_nonhomogeneous T-L}). Their norm involves a non‑factorable mixture of $l^q$ and $L^p$, forcing one to work directly with the full expression $\|\|\cdot\|_{l^q}\|_{L^p}$. This inherent complexity makes the Fefferman-Stein vector-valued maximal function inequality (Lemma  \ref{lem:Fguji}) not merely a handy tool, but an indispensable ingredient for establishing key estimates. In fact, it is the vector-valued counterpart of the classical maximal inequality (Lemma \ref{lem:max}) that provides the foundation for treating Triebel-Lizorkin spaces. Furthermore, in the construction of the commutator estimates (Theorem \ref{prop:commutator-estimates}), we rely on the Nikol'skij representations (Lemmas \ref{lem:Nik1} and \ref{lem:Nik2}), which play a crucial role in that argument.  They enable us to relax the regularity condition on the function from $s>0$ to $s>-d\min(\frac{1}{p},\frac{1}{p'})$, where $p'$ is the conjugate exponent of $p$. Moreover, under this condition, we obtain a more concise estimate for $s<1+\frac{d}{p}$, a result that was not present in previous works. While the focus of this paper is on the more challenging Triebel-Lizorkin setting, the framework developed here is equally applicable to Besov spaces, in that context, the same arguments hold with the scalar maximal inequality (Lemma \ref{lem:max}) replacing its vector-valued counterpart (Lemma \ref{lem:Fguji}). More precisely, we have the following results.
	\begin{theorem}\label{prop:commutator-estimates}
		(Commutator estimates) Let $d\in \mathbb{N}^+$, $1\leq p,q\leq \infty$, with the restriction that in the Triebel-Lizorkin case we take either $p<\infty$ or $p=q=\infty$, and $\frac{1}{p}+\frac{1}{p'}=1$. Assume that
		\begin{equation}
			s>-d\min\big(\frac{1}{p},\frac{1}{p'}\big)\quad \text{or}\quad s>-1-d\min\big(\frac{1}{p},\frac{1}{p'}\big)\quad\text{if}\quad \textup{div} v=0.
		\end{equation}
		Then there exists a constant $C$ depending only on $d, p, q$ and $s$ such that the following estimates hold true:
		\begin{enumerate}[label=(\roman*)]
			\item if $s<1+\frac{d}{p}$, then 
			\begin{align}
				\Big\|\big\|2^{js}([v,\Delta_j]\cdot \nabla f)\big\|_{l^q}\Big\|_{L^p(\mathbb{R}^d)}&\leq C\|\nabla v\|_{F^{\frac{d}{p}}_{p,\infty}(\mathbb{R}^d)\cap L^\infty(\mathbb{R}^d)}\|f\|_{F^s_{p,q}(\mathbb{R}^d)},\label{eq:commutator-estimates1}\\
				\Big\|2^{js}\big\|[v,\Delta_j]\cdot \nabla f\big\|_{L^p(\mathbb{R}^d)}\Big\|_{l^q}&\leq C\|\nabla v\|_{B^{\frac{d}{p}}_{p,\infty}(\mathbb{R}^d)\cap L^\infty(\mathbb{R}^d)}\|f\|_{B^s_{p,q}(\mathbb{R}^d)}.\label{eq:commutator-estimates1-Besov}
			\end{align}
			\item if $s>1+\frac{d}{p}$, or $s=1+\frac{d}{p}$ and $p=1$, then
			\begin{equation}\label{eq:commutator-estimates2}
				\Big\|\big\|2^{js}([v,\Delta_j]\cdot \nabla f)\big\|_{l^q}\Big\|_{L^p(\mathbb{R}^d)}\leq C\|\nabla v\|_{F^{s-1}_{p,q}(\mathbb{R}^d)}\|f\|_{F^s_{p,q}(\mathbb{R}^d)}.
			\end{equation}
			\item if $s>1+\frac{d}{p}$, or $s=1+\frac{d}{p}$ and $q=1$, then
			\begin{equation}\label{eq:commutator-estimates2-Besov}
				\Big\|2^{js}\big\|[v,\Delta_j]\cdot \nabla f\big\|_{L^p(\mathbb{R}^d)}\Big\|_{l^q}\leq C\|\nabla v\|_{B^{s-1}_{p,q}(\mathbb{R}^d)}\|f\|_{B^s_{p,q}(\mathbb{R}^d)}.
			\end{equation}
		\end{enumerate}
		Further, if $s>0$ (or $s>-1$ and $\textup{div}v=0$), then
		\begin{align}
			\Big\|\big\|2^{js}([v,\Delta_j]\cdot \nabla f)\big\|_{l^q}\Big\|_{L^p(\mathbb{R}^d)}&\leq C\big(\|\nabla v\|_{L^\infty(\mathbb{R}^d)}\|f\|_{F^s_{p,q}(\mathbb{R}^d)}+\|\nabla f\|_{L^\infty(\mathbb{R}^d)}\|\nabla v\|_{F^{s-1}_{p,q}(\mathbb{R}^d)}\big),\label{eq:commutator-estimates4}\\
			\Big\|\big\|2^{js}([v,\Delta_j]\cdot \nabla f)\big\|_{l^q}\Big\|_{L^p(\mathbb{R}^d)}&\leq C\big(\|\nabla v\|_{L^\infty(\mathbb{R}^d)}\|f\|_{F^s_{p,q}(\mathbb{R}^d)}+\|f\|_{L^\infty(\mathbb{R}^d)}\|\nabla v\|_{F^{s}_{p,q}(\mathbb{R}^d)}\big),\label{eq:commutator-estimates5}\\
			\Big\|2^{js}\big\|[v,\Delta_j]\cdot \nabla f\big\|_{L^p(\mathbb{R}^d)}\Big\|_{l^q}&\leq C\big(\|\nabla v\|_{L^\infty(\mathbb{R}^d)}\|f\|_{B^s_{p,q}(\mathbb{R}^d)}+\|\nabla f\|_{L^\infty(\mathbb{R}^d)}\|\nabla v\|_{B^{s-1}_{p,q}(\mathbb{R}^d)}\big),\label{eq:commutator-estimates4-Besov}\\
			\Big\|2^{js}\big\|[v,\Delta_j]\cdot \nabla f\big\|_{L^p(\mathbb{R}^d)}\Big\|_{l^q}&\leq C\big(\|\nabla v\|_{L^\infty(\mathbb{R}^d)}\|f\|_{B^s_{p,q}(\mathbb{R}^d)}+\|f\|_{L^\infty(\mathbb{R}^d)}\| \nabla v\|_{B^{s}_{p,q}(\mathbb{R}^d)}\big).\label{eq:commutator-estimates5-Besov}
		\end{align}
		If besides $f=v$, when $s>0$ (or $s>-1$ if $\textup{div}v=0$), then we also have 
		\begin{align}
			\Big\|\big\|2^{js}([v,\Delta_j]\cdot \nabla f)\big\|_{l^q}\Big\|_{L^p(\mathbb{R}^d)}&\leq C\|\nabla v\|_{L^\infty(\mathbb{R}^d)}\|f\|_{F^s_{p,q}(\mathbb{R}^d)},\label{eq:commutator-estimates3}\\
			\Big\|2^{js}\big\|[v,\Delta_j]\cdot \nabla f\big\|_{L^p(\mathbb{R}^d)}\Big\|_{l^q}&\leq C\|\nabla v\|_{L^\infty(\mathbb{R}^d)}\|f\|_{B^s_{p,q}(\mathbb{R}^d)}.\label{eq:commutator-estimates3-Besov}
		\end{align}
	\end{theorem}
	\begin{remark}
		The corresponding arguments in \cite{MR2768550,MR1827098} relies heavily on the specific structure of Besov spaces, where the $L^p$ norm is taken before the $l^q$ norm, together with Bernstein’s inequalities (Lemma \ref{lem:Bernstein}) and an integration by parts technique, to obtain precisely the Besov space version of the commutator estimate stated in Theorem \ref{prop:commutator-estimates}. This approach, however, does not extend directly to the Triebel-Lizorkin setting. In contrast, by employing the Nikol'skij representations together with the Hardy-Littlewood maximal function and Lemma \ref{lem:Sjguji}, we are able to recover the same estimates in a unified way that works for both Triebel-Lizorkin and Besov norms, thereby establishing the full statement of Theorem \ref{prop:commutator-estimates},
		which seems to be new in the existing literature.
	\end{remark}
	\begin{remark}
		There are a number of variations on the statement of Theorem \ref{prop:commutator-estimates}. For instance, the inequalities \eqref{eq:commutator-estimates1}, \eqref{eq:commutator-estimates1-Besov}, \eqref{eq:commutator-estimates4}, \eqref{eq:commutator-estimates4-Besov}, \eqref{eq:commutator-estimates3} and \eqref{eq:commutator-estimates3-Besov} also hold in the homogeneous framework (i.e., with $\dot{\Delta}_j$ in place of $\Delta_j$ and with homogeneous Triebel-Lizorkin and Besov norms instead of non\-homogeneous ones). The proof follows the same lines as that of Theorem \ref{prop:commutator-estimates}, it suffices to replace the nonhomogeneous blocks by their homogeneous counterparts.
	\end{remark}
	
	\begin{remark}
		Compared with the commutator estimates only in the homogeneous Triebel-Lizorkin spaces obtained in \cite{MR2592288,MR4240785,2601.10071}, our construc\-tion is more adaptable to situations involving lower regularity of $f$ or the absence of gradient control (see the comparison graph below). Specifically, the results in \cite{MR2592288,MR4240785,2601.10071} require $s>0 $ (or $s>-1$ and $ \textup{div}v=0 $), whereas we obtain an improved estimate under the weaker assumptions $s>-d\min(\frac{1}{p},\frac{1}{p'})$ (or $s>-1-d\min(\frac{1}{p},\frac{1}{p'})$ and $\textup{div}v=0$). Moreover, in the regime $s<1+\frac{d}{p}$, we derive a simpler estimate \eqref{eq:commutator-estimates1}, which avoids extra derivative terms and complicated index relations, a feature not present in the previous works.
		
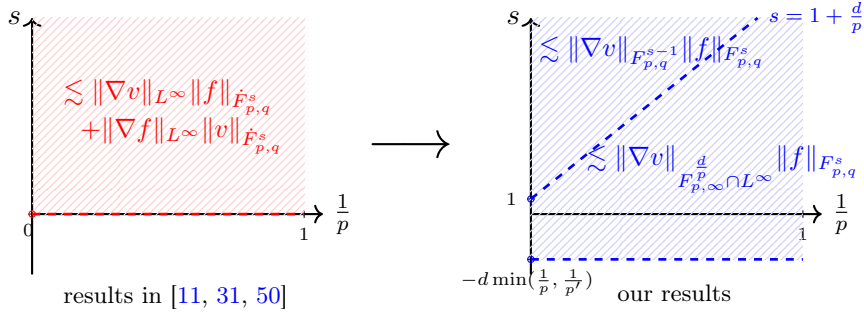
\[ \begin{tikzpicture}[
			scale=1,
			point/.style={circle, fill=black, inner sep=0.6pt}
			]
			% 第一部分：原始条件
			\begin{scope}[local bounding box=original]
				% 坐标轴
				\draw[->, thick] (0,0) -- (3.85,0) node[right] {$\frac{1}{p}$};
				\draw[->, thick] (0,-0.8) -- (0,2.6) node[left] {$s$};
				
				% 刻度
				\draw (0.05,0) -- (-0.05,0) node[below, font=\tiny] {$0$};
				\draw (3.6,0.05) -- (3.6,-0.05) node[below, font=\tiny] {$1$};

				% 原始条件区域
				\fill[pattern=north east lines, pattern color=red!20] 
				(0.0,0.0) rectangle (3.6,2.6);
				
				% 在区域中心添加文字 
				\node[red, align=center, font=\small] at (1.8, 1.3) {$\lesssim\|\nabla v\|_{L^\infty}\|f\|_{\dot{F}^s_{p,q}}$\\$\quad+\|\nabla f\|_{L^\infty}\|v\|_{\dot{F}^s_{p,q}}$};
				
				% 线段: 
				\draw[red, line width=1pt, dashed] (0,0) -- (3.6,0);
				
				% 点: 
				\draw[red, fill=none, thin] (0,0) circle (1.2pt);

				% 标题放在坐标轴下方
				\node[below, align=center, font=\footnotesize] at (1.9,-0.8) 
				{\textup{results in \cite{MR2592288,MR4240785,2601.10071}}};
			\end{scope}
			
			% 第二部分：拓展后条件
			\begin{scope}[local bounding box=extended, xshift=6.6cm]
				% 坐标轴
				\draw[->, thick] (0,0) -- (3.85,0) node[right] {$\frac{1}{p}$};
				\draw[->, thick] (0,-0.8) -- (0,2.6) node[left] {$s$};
				
				% 刻度
				\draw (0.05,0.2) -- (-0.05,0.2) node[left, font=\tiny] {$1$};
				\draw (3.6,0.05) -- (3.6,-0.05) node[below, font=\tiny] {$1$};
				\draw (0.05,-0.6) -- (-0.05,-0.6) node[below, font=\tiny] {$-d\min(\frac{1}{p},\frac{1}{p'})$};
				
				% 点: 
				\draw[blue, fill=none, thin] (0,0.2) circle (1.2pt);
				\draw[blue, fill=none, thin] (0,-0.6) circle (1.2pt);
				
				\fill[pattern=north east lines, pattern color=blue!20] 
				(0,-0.6) rectangle (3.6,2.6);
				
				% 线
				\draw[blue, line width=1pt, dashed] (0,0.2) -- (3.0,2.6) node[right, font=\scriptsize] {$s=1+\frac{d}{p}$}; 
				\draw[blue, line width=1pt, dashed] (0,-0.6) -- (3.6,-0.6); 
				\node[blue, below,align=center, font=\small] at (1.6, 2.5) {$\lesssim\|\nabla v\|_{F^{s-1}_{p,q}}\|f\|_{F^s_{p,q}}$};
				\node[blue, right,align=center, font=\small] at (0.6, 0.6) {$\lesssim\|\nabla v\|_{F^{\frac{d}{p}}_{p,\infty}\cap L^\infty}\|f\|_{F^s_{p,q}}$};
				
				% 标题放在坐标轴下方
				\node[below, align=center, font=\footnotesize] at (1.9,-0.8) 
				{\textup{our results}};
			\end{scope}
			
			% 连接两部分的箭头 - 使用数学符号\Rightarrow
			\node[font=\small] at (5.0, 0.88)  {\scalebox{2}{$\longrightarrow$}};

		\end{tikzpicture} \]
		The improvement is achieved by a more refined treatment of Bony's para-product decomposition. To ensure that only the gradient part of $v$ appears in the estimates, we decompose $v$ into its low and high frequency parts by writing $v=S_1v+\tilde{v}$. For the term $ \Delta_jR(\tilde{v},\partial_if)$, we split it into $\Delta_j\partial_iR(\tilde{v},f)$ and $\Delta_jR(\textup{div}\tilde{v},f)$ and estimate them separately. In particular, applying the Nikol'skij representations (Lemma \ref{lem:Nik2}) to both components allows us to relax the regularity index up to $s>-d\min(\frac{1}{p},\frac{1}{p'})$ (or $s>-1-d\min(\frac{1}{p},\frac{1}{p'})$ and $\textup{div}v=0$). On the other hand, for the estimate of $\Delta_jT_{\partial_if}\tilde{v}^i$, we combine Lemmas \ref{lem:Nik1} and \ref{lem:Sjguji} to handle the case $s<1+\frac{d}{p}$. The key point is that in treating the terms under consideration, we employ the Nikol'skij representations (Lemmas \ref{lem:Nik1} and \ref{lem:Nik2}) instead of Fefferman-Stein vector-valued maximal inequality (Lemma \ref{lem:Fguji}) used in the aforementioned works, this avoids the emergence of an additional gradient term $\nabla f$.
	\end{remark}
	
	\subsection{Applications.}
	With the commutator estimates established, we observe that they play a crucial role in deriving a priori estimates for following transport equations:
	\begin{equation}\label{Eq}
		\begin{cases}
			\partial_tf+v\cdot\nabla f=g,& t > 0, x \in \mathbb{R}^d,\\
			f|_{t=0}=f_0,
		\end{cases}\tag{T}
	\end{equation}
	where $v\colon\mathbb{R}^+ \times \mathbb{R}^d \to \mathbb{R}^d$ is a given time-dependent vector field, $f_0\colon \mathbb{R}^d \to \mathbb{R}^N$ and $g\colon \mathbb{R}^+ \times \mathbb{R}^d \to \mathbb{R}^N$ $(d,N\in\mathbb{N}^+)$ are the given initial data and source term, respectively. This naturally leads to the second objective of this paper, namely, to develop a theory for transport equations in Triebel-Lizorkin spaces. Transport equations of this type arise naturally in many areas of mathematical physics, particularly in the analysis of partial differential equations from fluid mechanics. Although in such contexts the velocity field $v$ and the source $g$ often depend nonlinearly on the unknown $f$, having a good theory for the linear transport equations \eqref{Eq} constitutes an essential preliminary step for studying more complex coupled systems. Over the past few decades, significant progress has been made on the well-posedness of transport equations with weakly differentiable (in the space variables) velocity fields. In their seminal work, DiPerna and Lions \cite{MR1022305} established existence and uniqueness of solutions in $L^\infty$ under the assumption that $v$ belongs to $W^{1,1}$ (with bounded divergence).  Ambrosio \cite{MR2096794} later extended the result to BV vector fields, under additional conditions on the divergence. Chemin \cite{MR1688875} established a priori estimates for the transport equation in H\"{o}lder spaces $\mathcal{C}^r$, under the condition that the velocity field $v$ is divergence-free. A systematic well-posedness theory for \eqref{Eq} in Besov spaces has been developed in \cite{MR2768550,MR2231013}, provided the vector field $v$ is at least Lipschitz in the spatial variable. Recently, the well-posedness of transport equations in the Triebel-Lizorkin spaces was established in \cite{2601.10071} for the case $s>1+\frac{d}{p}$. When studying the well-posedness of transport equations in different function spaces, the Triebel-Lizorkin and Besov spaces are often treated separately. This is because their proofs rely on distinct intrinsic structures of 
	$B^s_{p,q}$ and $F^s_{p,q}$ and thus require different methods. However, both spaces are constructed based on the Littlewood-Paley frequency decomposition, which decomposes a function into a series of frequency-localized smooth pieces. As a result, they share many similar properties (see Lemma \ref{lem:Triebel-Lizorkin-properties}). It is therefore desirable to unify the discussion of well-posedness for transport type equations in these two spaces. Here and throughout the paper, for simplicity, we will use $X^s_{p,q}$ (resp., $\dot{X}^s_{p,q}$, ) to denote $F^s_{p,q}$ (resp., $\dot{F}^s_{p,q}$) or $B^s_{p,q}$ (resp., $\dot{B}^s_{p,q}$). Using this notation, we develop the following general theory\footnote{Our results (Theorems \ref{Thm:a-priori-estimates} and \ref{thm:solve-transport}) for $x\in\mathbb{R}^d$ may be easily carried out to the periodic case $x\in\mathbb{T}^d$.} (Theorems \ref{Thm:a-priori-estimates} and \ref{thm:solve-transport}) for the transport equations \eqref{Eq} in $X^s_{p,q}$ by employing our commutator estimates (Theorem \ref{prop:commutator-estimates}). 
	\begin{theorem}\label{Thm:a-priori-estimates}
		(A priori estimates for transport equations) Let $d\in\mathbb{N}^+$, $1\leq p,q\leq \infty$, with the restriction that in the Triebel-Lizorkin case we take either $p<\infty$ or $p=q=\infty$, and $\frac{1}{p}+\frac{1}{p'}=1$. Assume that 
		\begin{equation}
			s>-d\min\big(\frac{1}{p},\frac{1}{p'}\big) \quad\text{or}\quad s>-1-d\min\big(\frac{1}{p},\frac{1}{p'}\big)\quad\text{if}\quad\textup{div}v=0.
		\end{equation}
		Let $v$ be a vector field such that $\nabla v$ belongs to $L^1(0,T;X^{s-1}_{p,q}(\mathbb{R}^d))$ if $s>1+\frac{d}{p}$ or $s=1+\frac{d}{p}$ and $p=1 $ if $X^{s-1}_{p,q}(\mathbb{R}^d)=F^{s-1}_{p,q}(\mathbb{R}^d)$ (resp. $q=1$ if $X^{s-1}_{p,q}(\mathbb{R}^d)=B^{s-1}_{p,q}(\mathbb{R}^d) $), or to $L^1(0,T;X^{d/p}_{p,\infty}(\mathbb{R}^d)\cap L^\infty(\mathbb{R}^d))$ if $s<1+\frac{d}{p}$. Then there exists a constant $C$ depending only on $d,p,q$ and $s$, such that for all solutions $f\in L^\infty(0,T;X^s_{p,q}(\mathbb{R}^d))$ to \eqref{Eq} with initial data $f_0\in X^s_{p,q}(\mathbb{R}^d)$
		and $g \in L^1(0,T;X^s_{p,q}(\mathbb{R}^d))$, we have
		\begin{equation}
			\|f(t)\|_{X^s_{p,q}}\leq \|f_0\|_{X^s_{p,q}}+\int_{0}^t\|g(\tau)\|_{X^s_{p,q}}d\tau+C\int_{0}^tZ(\tau)\|f(\tau)\|_{X^s_{p,q}}d\tau,
		\end{equation}
		or hence, 
		\begin{equation}\label{eq:a-priori-estimates}
			\|f(t)\|_{X^s_{p,q}(\mathbb{R}^d)}\leq
			e^{C\int_{0}^tZ(\tau)d\tau} \Big(\|f_0\|_{X^s_{p,q}(\mathbb{R}^d)}+\int_0^te^{-C\int_{0}^\tau Z(\tau')d\tau'}\|g(\tau)\|_{X^s_{p,q}(\mathbb{R}^d)}d\tau\Big),
		\end{equation}
		with
		\begin{equation*}
			Z(t)=\begin{cases}
				\|\nabla v\|_{X^{\frac{d}{p}}_{p,\infty}(\mathbb{R}^d)\cap L^\infty(\mathbb{R}^d)},\ &\text{if}\quad s<1+\frac{d}{p},\\
				\|\nabla v\|_{X^{s-1}_{p,q}(\mathbb{R}^d)},\ &\text{if}\quad s>1+\frac{d}{p}\ \text{or}\ s=1+\frac{d}{p}\ \text{and}\ \begin{cases*}
					p=1,\ \text{if}\ X^{s-1}_{p,q}(\mathbb{R}^d)=F^{s-1}_{p,q}(\mathbb{R}^d),\\
					q=1,\ \text{if}\ X^{s-1}_{p,q}(\mathbb{R}^d)=B^{s-1}_{p,q}(\mathbb{R}^d).
				\end{cases*}\
			\end{cases}
		\end{equation*}
		If $f=v$, then for all $s>0$ $(s>-1$ if $\textup{div}v=0)$ estimates \eqref{eq:a-priori-estimates} hold with 
		\[ Z(t)=\|\nabla v\|_{L^\infty(\mathbb{R}^d)}.  \]
	\end{theorem}
	\begin{theorem}\label{thm:solve-transport}
		(Local well-posedness for transport equations) Let $d,p,q$ and $s$ be as in the statement of Theorem \ref{Thm:a-priori-estimates}. Let $f_0 \in X^s_{p,q}(\mathbb{R}^d)$ and $g\in L^1(0,T;X^s_{p,q}(\mathbb{R}^d))$. Let v be a time dependent vector field with coefficients in $L^\rho(0,T;X^{-M}_{\infty,\infty}(\mathbb{R}^d))$ for some $\rho>1$ and $M>0$, and such that $\nabla v\in L^1(0,T;X^{\frac{d}{p}}_{p,\infty}(\mathbb{R}^d)\cap L^\infty(\mathbb{R}^d))$ if $s<1+\frac{d}{p}$, and $\nabla v\in L^1(0,T;X^{s-1}_{p,q}(\mathbb{R}^d))$ if $s>1+\frac{d}{p}$ or $s=1+\frac{d}{p}$ and $p=1 $ if $X^{s-1}_{p,q}(\mathbb{R}^d)=F^{s-1}_{p,q}(\mathbb{R}^d)$ (resp. $q=1$ if $X^{s-1}_{p,q}(\mathbb{R}^d)=B^{s-1}_{p,q}(\mathbb{R}^d) $). Then the transport equations \eqref{Eq} has a unique solution $f\in L^\infty([0,T];X^s_{p,q}(\mathbb{R}^d))\bigcap(\cap_{s'<s}C([0,T];X^{s'}_{p,1}(\mathbb{R}^d))) $ and the estimates in Theorem \ref{Thm:a-priori-estimates} hold true. If moreover $q<\infty$, then we have $f\in C([0,T];X^s_{p,q}(\mathbb{R}^d))$. Furthermore, the data-to-solution map $f_0\mapsto f$ is continuous from $X^s_{p,q}(\mathbb{R}^d)$ into $L^\infty ([0,T];X^{s}_{p,q}(\mathbb{R}^d))\cap C([0,T];X^{s'}_{p,1}(\mathbb{R}^d))$ for every $s'<s$ if $q=\infty$, and into $C ([0,T];X^{s}_{p,q}(\mathbb{R}^d))$ if $q<\infty$.
	\end{theorem}
	\begin{remark}
		Compared with the related results for $F^s_{p,q}$ in \cite{2601.10071}, our findings (Theorems \ref{Thm:a-priori-estimates} and \ref{thm:solve-transport}) cover not only the case $s>1+\frac{d}{p}$ treated there but also the case $s\leq 1+\frac{d}{p}$. The proof in \cite{2601.10071} relies on the fact that $\|f\|_{F^s_{p,q}}=\|f\|_{L^p}+\|f\|_{\dot{F}^s_{p,q}}$ when $s>0$ and proceeds by working on the space $\dot{F}^s_{p,q}$, which forces the regularity condition 
		$s>0$. In contrast, we work directly on $F^s_{p,q}$ and, by applying Theorem \ref{prop:commutator-estimates}, obtain improved regularity, namely $s>-d\min(\frac{1}{p},\frac{1}{p'})$.
	\end{remark}
	\begin{remark}
		Our objective is not merely to establish an improved result in Triebel-Lizorkin spaces, but to provide a unified perspective. More precisely, using the notation $X^s_{p,q}$, we unify the result obtained in $F^s_{p,q}$ with that in $B^s_{p,q}$ from \cite{MR2768550,MR1827098}. This unified framework suggests that previous results for certain transport type equations, such as the Euler-Poincaré system and Camassa-Holm equation, established in $B^s_{p,q}$, may potentially be adapted to $F^s_{p,q}$, which we will consider below.
	\end{remark}
	Note that our results for transport equations can be applied to the ideal fluids, such as the incompressible Euler equation and the ideal MHD system, which have been studied in \cite{MR1880646,MR2020259,MR2592288,MR4240785,MR951744,2601.10071}. In this paper, as an application, we consider the two-component Euler-Poincaré equations which can be recast in the form of transport equations via a suitable transformation. It is worth pointing out that if one directly applies the transport result from our recent work \cite{2601.10071} to study the Cauchy problem for the two-component Euler-Poincaré system, it is readily seen that the regularity condition $s>2+\frac{d}{p}$ is obtained, which is not the optimal index $1+\frac{d}{p}$. We therefore take this system as an illustration to examine the performance of our transport theory (Theorems \ref{Thm:a-priori-estimates} and \ref{thm:solve-transport}). We aim to investigate the well-posedness and blow-up phenomena for this system. The two-component Euler-Poincar\'{e} system is given as follows:
	\begin{equation}\label{eq}
		\begin{cases}
			\partial_tm+u\cdot\nabla m+(\nabla u)^T\cdot m+m(\text{div}u)+g\rho\nabla\rho=0,\\
			\partial_t\rho+\text{div}(\rho u)=0,\\
			m=(I-\Delta)u.\\
		\end{cases}
	\end{equation}
	where the vector fields $u=u(t,x)$ and $m=m(t,x)$ are defined from $\mathbb{R}^+\times\mathbb{R}^d\to\mathbb{R}^d$, the scalar function $\rho=\rho(t,x)$ is defined from $\mathbb{R}^+\times\mathbb{R}^d\to\mathbb{R}$,  $d\in\mathbb{N}^+$ is the spatial dimensions, and $g>0$ is the downward constant acceleration of gravity. The system \eqref{eq}  was introduced in \cite{MR2552212,MR2904389} as a framework for modeling and analyzing fluid dynamics, particularly in the contexts of nonlinear shallow water waves, geophysical fluids, and turbulence modeling, or alternatively, as a reformulation of geodesic flow on the diffeomorphism group. The Cauchy problem for system \eqref{eq} was studied in \cite{MR3255472}, where the local well-posedness, blow-up criteria in certain Sobolev spaces, and some radially symmetric finite-time blow-up solutions were investigated. Moreover, system \eqref{eq} is closely related to the following incompressible $\alpha$-Euler equation in fluid mechanics:
	\begin{equation}\label{eq:equation68}
		\begin{cases}
			\partial_tm+u\cdot\nabla m+(\nabla u)^T\cdot m =-\nabla P,\\
			\textup{div}u=0,\\
			m=(1-\alpha^2\Delta)u.
		\end{cases}
	\end{equation}
	For $\alpha\to0$, system \eqref{eq:equation68} converges to the classical Euler equation \cite{MR1627802}.
	
	When the system is decoupled (i.e., formally, $\rho\equiv0$), system \eqref{eq} reduces to the Euler-Poincar\'{e} equations associated with the diffeomorphism group in \cite{MR2103008}, that is
	\begin{equation}\label{eq:equation65}
		\partial_tm+u\cdot\nabla m+(\nabla u)^T\cdot m+m(\text{div}u)=0\quad \text{with}\quad m=(I-\Delta)u,
	\end{equation}
	which was first proposed exactly in the way that a class of its singular solutions generalize the peakon solutions of the classical Camassa-Holm equation to higher spatial dimensions \cite{MR2031278}. The local well-posedness, global existence, and blow-up of solutions of the Cauchy problem for \eqref{eq:equation65} have been discussed in \cite{MR2964772,MR3116009,MR3277196}. 
	
	In particular, system \eqref{eq} becomes the following completely integrable two-component Camassa-Holm (2CH) shallow water system as $d=1$:
	\begin{equation}\label{eq:equation66}
		\begin{cases}
			m_t+um_x+2u_xm+g\rho\rho_x=0,\\
			\rho_t+(u\rho)_x=0,\\
			m=(1-\partial_x^2)u.
		\end{cases}
	\end{equation}
	The 2CH system in \eqref{eq:equation66} was derived in the context of shallow water theory \cite{MR2474608} from
	the Green-Naghdi equations or incompressible Euler equation with constant vorticity
	\cite{MR2598636}. Some studies on the Cauchy problem for system \eqref{eq:equation66} can be seen in \cite{MR2474608,MR3005543,MR2609545}, for instance. 
	
	In the case that $d=1$ and $\rho\equiv0$, system \eqref{eq} becomes the celebrated Camassa-Holm (CH) equation as follows:
	\begin{equation}\label{eq:equation67}
		u_t-u_{xxt}+3uu_x=2u_xu_{xx}+uu_{xxx},
	\end{equation}
	which is proposed as a bi-Hamiltonian equation \cite{MR636470} and models the unidirectional propagation of shallow water waves over a flat bottom \cite{MR1234453} (see also a rigorous justification in shallow water \cite{MR2481064}). Its solitary waves are peaked solitons (peakons) \cite{MR1234453} (see also \cite{MR2257390,MR2753609}
	 for the waves of great height in irrotational water waves), and they are orbitally stable \cite{MR1854962,MR1737505}. The Cauchy problem for the CH equation have been studied extensively. Its local well-posedness has been shown in \cite{MR1604278,MR1827098,MR1990847}. Moreover, it has both global strong solutions and blow-up solutions within finite time \cite{MR1775353,MR1631589,MR1668586,MR1604278} which is featured as wave breaking \cite{MR1668586} (namely, the wave
	remains bounded while its slope becomes unbounded infinite time \cite{MR483954}). In addition, it possesses global weak solutions; see the discussions in \cite{MR2278406,MR4108918,MR1773414}, for instance. 
	
	Now, let us set up the Cauchy problem for system \eqref{eq}. Set $\gamma=\rho-1$ and $g=1$ for convenience. According to \cite{MR4405195}, we can transform system \eqref{eq} into the following nonlocal transport form:
	\begin{equation}\label{transport-system}
		\begin{cases}
			\partial_tu+u\cdot\nabla  u=F_1(u,\gamma),\quad&(t,x)\in\mathbb{R}^+\times \mathbb{R}^d,\\
			\partial_t\gamma+u\cdot\nabla  \gamma=F_2(u,\gamma),\quad&(t,x)\in\mathbb{R}^+\times \mathbb{R}^d,\\
			u(0,x)=u_0(x),\quad &\ \ \quad x\in\mathbb{R}^d,\\
			\gamma(0,x)=\gamma_0(x)=\rho_0(x)-1,\quad&\ \ \quad x\in\mathbb{R}^d,
		\end{cases}
	\end{equation}
	where
	\begin{align}
		F_1(u,\gamma)\triangleq&-(I-\Delta)^{-1}\text{div}\Big(\nabla u\nabla u+\nabla u\nabla u^T-\nabla u^T\nabla u-\nabla u(\text{div}u)+\frac{1}{2}|\nabla u|^2I+\frac{1}{2}\gamma^2I+\gamma I\Big)\nonumber\\
		&-(I-\Delta)^{-1}\Big(u(\text{div}u)+\nabla u^T\cdot u\Big),\label{eq:F1}
	\end{align}
	and
	\begin{equation}\label{eq:F2}
		F_2(u,\gamma)\triangleq-\gamma\textup{div}u-\text{div}u.
	\end{equation}
	Recall that the well-posedness and blow-up criterion for \eqref{transport-system} in Besov spaces have been established in \cite{MR4405195}. However, their proof does not extend directly to the Triebel-Lizorkin setting, and we therefore consider only this case here. To this end, we define some functional spaces.
	\begin{definition}
		For $T>0,s\in\mathbb{R}$, and $(p,q)\in[1,\infty)\times[1,\infty]$ or $p=q=\infty$, we define the space $E^s_{p,q}(T)$ as follows:
		\[ 
		E^s_{p,q}(T)\triangleq
		\begin{cases}
			C([0,T];F^s_{p,q}(\mathbb{R}^d))\cap C^1([0,T];F^{s-1}_{p,q}(\mathbb{R}^d)),\quad\text{if}\quad q<\infty,\\
			L^\infty(0,T;F^s_{p,\infty}(\mathbb{R}^d))\cap \textup{Lip}(0,T;F^{s-1}_{p,\infty}(\mathbb{R}^d)),\quad\text{if}\quad q=\infty.
		\end{cases} 
		\]
	\end{definition} 
	We are now in a position to state our well-posedness and blow-up criterion results for \eqref{transport-system} as follows.
	\begin{theorem}\label{thm:well-poseness}
		Let $d\in\mathbb{N}^+$, $(p,q)\in[1,\infty)\times[1,\infty]$ or $p=q=\infty$, and $s>\max(1+\frac{d}{p},\frac{3}{2})$ (or $s=1+d$ with $p=1$ and $d\geq 2$). Suppose that $(u_0,\gamma_0)\in F^{s}_{p,q}(\mathbb{R}^d)\times F^{s-1}_{p,q}(\mathbb{R}^d)$. Then there exists a time $T>0$ such that  $(u,\gamma)\in E^{s}_{p,q}(T)\times E^{s-1}_{p,q}(T)$ is the unique solution to system \eqref{transport-system}, and the solution depends continuously on the initial data, that is, the mapping $(u_0,\gamma_0)\mapsto (u,\gamma)$ is continue from $D^{s}(R)\times D^{s-1}(R)$ into $E^{s'}_{p,q}(T)\times E^{s'-1}_{p,q}(T)$ for all $s'<s$ if $q=\infty$, and $s'=s$ otherwise, where $D^s(R)\triangleq \{f\in F^s_{p,q}(\mathbb{R}^d):\|f\|_{F^s_{p,q}}\leq R\}$.  
	\end{theorem}
	\begin{theorem}\label{thm:Blow-up-criterion-2}
		Let $d\in\mathbb{N}^+$, and $(p,q)\in[1,\infty)\times[1,\infty]$ or $p=q=\infty$. Suppose that $(u_0,\gamma_0)\in F^s_{p,q}(\mathbb{R}^d)\times F^{s-1}_{p,q}(\mathbb{R}^d)$ with $s>\max(1+\frac{d}{p},\frac{3}{2})$ (or $s=1+d$ with $p=1$ and $d\geq 2$) and $(u,\gamma)$ is the corresponding solution to system \eqref{transport-system}. If the solution $(u,\gamma)$ blows up in finite time (i.e. the lifespan of solution $T^\star<\infty$), then
		\begin{equation}\label{eq:equation69}
			\int_{0}^{T^\star}\|\nabla u(\tau)\|_{L^\infty}d\tau=\infty.
		\end{equation}
	\end{theorem}
	\begin{remark}
		The blow-up criterion \eqref{eq:equation69} only involves $\nabla u$, which is independent of the second density component $\gamma$ or $\rho$. Moreover, our results Theorems \ref{thm:well-poseness}-\ref{thm:Blow-up-criterion-2} in both sub-critical and critical Triebel-Lizorkin spaces recover the corresponding results in H\"{o}lder spaces $\mathcal{C}^r$ and Sobolev spaces $W^{s,p}$, and hence extend the related results in \cite{MR2964772,MR1604278,MR1827098,MR3255472,MR2609545,MR3116009,MR4405195,MR3277196}.
	\end{remark}
	\begin{remark}
		The constructions of global data and blow-up data for system \eqref{transport-system} in Besov spaces, as carried out in \cite{MR4405195}, can be carried over to the Triebel-Lizorkin setting with the necessary modifications. For brevity, we omit the details here. Nevertheless, it is worth pointing out that, by arguments analogous to those used in the proofs of Theorems \ref{thm:well-poseness} and \ref{thm:Blow-up-criterion-2}, one can recover the corresponding results in Besov spaces by virtue of the general theory (Theorems \ref{Thm:a-priori-estimates} and \ref{thm:solve-transport}) established above. In contrast, the proof of the Besov-space results in \cite{MR4405195} does not directly apply to Theorems \ref{thm:well-poseness} and \ref{thm:Blow-up-criterion-2}, i.e., to the Triebel-Lizorkin case.
	\end{remark}
	\begin{remark}
		In Theorem \ref{thm:well-poseness}, the one-dimensional critical case, i.e., 
		$s=2,p=1$, cannot be directly handled by the transport theory developed in this paper, as can be seen from Remark \ref{remark:product-critical} for details. This case requires  more delicate estimates and will be addressed in a forthcoming paper.
	\end{remark}
	The remainder of the paper is organized as follows. In Section 2, we recall some facts on the Littlewood-Paley analysis, Triebel-Lizorkin and Besov spaces,  Nikol'skij representations, Hardy-Littlewood maximal functions,
	and Morse type inequalities as well. In Section 3, we derive the commutator estimates (Theorem \ref{prop:commutator-estimates}). In Sections 4 and 5, we establish a priori estimates (Theorem \ref{Thm:a-priori-estimates}) and the local well-posedness result (Theorem \ref{thm:solve-transport}) for the transport equations \eqref{Eq}, respectively. In Section 6, we apply Theorems \ref{Thm:a-priori-estimates} and \ref{thm:solve-transport} from Sections 4 and 5 to the two-component Euler-Poincaré system \eqref{transport-system},
	yielding its local well-posedness (Theorem \ref{thm:well-poseness}). Finally, in Section 7, we give the proof of blow-up criterion (Theorem \ref{thm:Blow-up-criterion-2}).

	\section{Preliminaries}\label{section:2}
	In this section, we provide some basic facts concerning the Littlewood-Paley decomposition, Triebel-Lizorkin spaces, Besov spaces and their useful properties. First, we introduce some notation related to the theory of Littlewood-Paley decomposition.
	
	Let $\mathscr{B}=\{\xi\in\mathbb{R}^d\colon|\xi|\leq \frac{4}{3} \}$ and $\mathscr{C}=\{\xi\in\mathbb{R}^d\colon\frac{3}{4}\leq |\xi|\leq \frac{8}{3} \}$, $(\varphi,\chi)$ be a couple of smooth functions valued in $[0,1]$, such that $\varphi$ is support in  $\mathscr{C}$, $\chi$ is supported in $\mathscr{B}$ and 
	\[ \forall \xi \in \mathbb{R}^d,\quad \chi(\xi)+\sum_{j\in\mathbb{N}}\varphi(2^{-j}\xi)=1. \]
	We denote $\varphi_j(\xi)=\varphi(2^{-j}\xi),h=\mathscr{F}^{-1}\varphi$ and $\tilde{h}=\mathscr{F}^{-1}\chi$. For $f\in\mathscr{S}'(\mathbb{R}^d)$, one can define nonhomogeneous dyadic blocks as follow:
	\begin{equation}
		\Delta_jf\triangleq0,\quad\text{if}\quad j\leq -2,
	\end{equation}
	\begin{equation}
		\Delta_{-1}f\triangleq\chi(D)f=\tilde{h}*f,
	\end{equation}
	\begin{equation}\label{def:Delta_j}
		\Delta_jf\triangleq\varphi(2^{-j}D)f=2^{jd}\int_{\mathbb{R}^d}h(2^jy)f(x-y)dy,\quad\text{if}\quad  j\geq 0,
	\end{equation}
	\begin{equation}\label{def:S_j}
		S_jf\triangleq\sum_{k\leq j-1}\Delta_kf=\chi(2^{-j}D)f=2^{jd}\int_{\mathbb{R}^d}\tilde{h}(2^jy)f(x-y)dy.
	\end{equation}
	It is easy to see $\Delta_j=S_j-S_{j-1}$. For all $f,g\in \mathscr{S}'(\mathbb{R}^d)$, we observe
	\begin{equation}\label{a.o.c}
		\Delta_j\Delta_kf\equiv 0\quad \text{if} \quad |j-k|\geq 2\quad\text{and}\quad\Delta_j(\Delta_{k-1}f\Delta_kg)\equiv0\quad\text{if}\quad k\leq j-1,
	\end{equation}
	\begin{equation}\label{a.o.c-1}
		S_j\Delta_kf\equiv0\quad\text{if}\quad j\leq k-1\quad\text{and}\quad \Delta_j(S_{k-1}f\Delta_kg)\equiv 0\quad \text{if} \quad |j-k|\geq 5.
	\end{equation}
	The homogeneous dyadic blocks $\dot{\Delta}_j$ and the homogeneous low-frequency cut
	off operators $\dot{S}_j$ are defined for all $j\in\mathbb{Z}$ by
	\begin{equation}\label{def:dot-Delta_j}
		\dot{\Delta}_jf\triangleq\varphi(2^{-j}D)f=2^{jd}\int_{\mathbb{R}^d}h(2^jy)f(x-y)dy,
	\end{equation}
	\begin{equation}\label{def:dot-S_j}
		\dot{S}_jf\triangleq\chi(2^{-j}D)f=2^{jd}\int_{\mathbb{R}^d}\tilde{h}(2^jy)f(x-y)dy.
	\end{equation}

	\begin{remark}
		For each $j\in\mathbb{Z}$, we have $\|\Delta_jf\|_{L^p},\|S_jf\|_{L^p}\leq C\|f\|_{L^p}$ for some positive constant C independent of $j$.
	\end{remark}
		\begin{lemma}\label{lem:Bernstein}
		\textup{\cite{MR2768550}}(Bernstein's inequalities) Let $\mathcal{B}$ be a ball and $\mathcal{C}$ an annulus. A constant $C>0$ exists such that for any $k\in\mathbb{N},1\leq p\leq q\leq \infty$, and any function $f\in L^p(\mathbb{R}^d)$, we have
		\[ \text{Supp}\hat{f}\subset\lambda\mathcal{B}\Rightarrow\|D^kf\|_{L^q}\triangleq\sup_{|\alpha|=k}\|\partial^\alpha f\|_{L^q}\leq C^{k+1}\lambda  ^{k+d(\frac{1}{p}-\frac{1}{q})}\|f\|_{L^p}, \] 
		\[ \text{Supp}\hat{f}\subset\lambda\mathcal{C}\Rightarrow C^{-k-1}\lambda^k\|f\|_{L^p}\leq\|D^kf\|_{L^p}\leq C^{k+1}\lambda  ^{k}\|f\|_{L^p}. \] 
	\end{lemma}
	
	With the introduction of Littlewood-Paley decomposition, let us recall the definitions of Triebel-Lizorkin space and Besov spaces \cite{MR781540}. Let $s\in\mathbb{R},(p,q)\in [1,\infty)\times[1,\infty]$ or $p=q=\infty$. The Triebel-Lizorkin space $F^s_{p,q}(\mathbb{R}^d)$\footnote{For the endpoint case $p=q=\infty$, $F^s_{\infty,\infty}$ is equipped with the norm
		\[ \|f\|_{F^s_{\infty,\infty}}=\sup_{j\in \mathbb{Z}}2^{js}\|\Delta_jf\|_{L^\infty}, \]
		which coincides with the Besov spaces $B^s_{\infty,\infty}$, the homogeneous case is analogous.} ($F^s_{p,q}$ for short) is defined by 
	\begin{equation}\label{def:norm_nonhomogeneous T-L}
		F^s_{p,q}(\mathbb{R}^d)=\Big\{f\in \mathscr{S}'(\mathbb{R}^d)\colon \|f\|_{F^s_{p,q}}=\Big\|\big\|2^{js}|\Delta_jf|\big\|_{l^q(j\geq -1)}\Big\|_{L^p}<\infty\Big\}.
	\end{equation}
	Let $s\in\mathbb{R},(p,q)\in [1,\infty]\times[1,\infty]$. The Besov space $B^s_{p,q}(\mathbb{R}^d)$($B^s_{p,q}$ for short) is defined by 
	\begin{equation}\label{def:norm_nonhomogeneous B}
		B^s_{p,q}(\mathbb{R}^d)=\Big\{f\in \mathscr{S}'(\mathbb{R}^d)\colon \|f\|_{B^s_{p,q}}=\Big\|2^{js}\big\|\Delta_jf\big\|_{L^p}\Big\|_{l^q(j\geq -1)}<\infty\Big\}.
	\end{equation}
	Denote $\dot{\mathscr{S}}'(\mathbb{R}^d)$ as the dual space of $\dot{\mathscr{S}}(\mathbb{R}^d)\triangleq\{f\in \mathscr{S}(\mathbb{R}^d)\colon\partial^\alpha\hat{f}(0)=0,\forall\text{ multi-index }\alpha\in\mathbb{N}^d \}$ with the Schwartz space $\mathscr{S}(\mathbb{R}^d)$. We now turn to the homogeneous versions of these spaces. Let $s\in\mathbb{R},(p,q)\in [1,\infty)\times[1,\infty]$ or $p=q=\infty$. The homogeneous Triebel-Lizorkin space $\dot{F}^s_{p,q}(\mathbb{R}^d)$ ($\dot{F}^s_{p,q}$ for short) is defined by 
	\begin{equation}\label{def:norm_homogeneous T-L}
		\dot{F}^s_{p,q}(\mathbb{R}^d)=\Big\{f\in \dot{\mathscr{S}}'(\mathbb{R}^d)\colon \|f\|_{\dot{F}^s_{p,q}}=\Big\|\big\|2^{js}|\dot{\Delta}_jf|\big\|_{l^q(j\in\mathbb{Z})}\Big\|_{L^p}<\infty\Big\}.
	\end{equation}
	Let $s\in\mathbb{R},(p,q)\in [1,\infty]\times[1,\infty]$. The homogeneous Besov space $\dot{B}^s_{p,q}(\mathbb{R}^d)$($\dot{B}^s_{p,q}$ for short) is defined by 
	\begin{equation}\label{def:norm_homogeneous B}
		\dot{B}^s_{p,q}(\mathbb{R}^d)=\Big\{f\in \dot{\mathscr{S}}'(\mathbb{R}^d)\colon \|f\|_{\dot{B}^s_{p,q}}=\Big\|2^{js}\big\|\dot{\Delta}_jf\big\|_{L^p}\Big\|_{l^q(j\in\mathbb{Z})}<\infty\Big\}.
	\end{equation}
	
	Let us give some classical properties for the Triebel-Lizorkin and Besov spaces. As mentioned in the Introduction, for convenience, we will use $X^s_{p,q}$ (resp., $\dot{X}^s_{p,q}$, ) to denote $F^s_{p,q}$ (resp., $\dot{F}^s_{p,q}$) or $B^s_{p,q}$ (resp., $\dot{B}^s_{p,q}$).
	\begin{lemma}\label{lem:Triebel-Lizorkin-properties}
	    \textup{\cite{MR2768550,MR1419319,MR781540,2601.10071}} The following properties hold:
		\begin{enumerate}[label={(\roman*)}]
		\item Embedding: Let $\varepsilon>0$ and suppose $q_0<q_1$, $p_0<p_1 $ and $s_0-\frac{d}{p_0}=s_1-\frac{d}{p_1}$. Then
		\[ X^s_{p,q_0}\hookrightarrow X^s_{p,q_1},\qquad X^{s+\varepsilon}_{p,q}\hookrightarrow X^s_{p,q},\qquad F^{s_0}_{p_0,\infty}\hookrightarrow F^{s_1}_{p_1,q},\qquad B^{s_0}_{p_0,q}\hookrightarrow B^{s_1}_{p_1,q}. \]
		\item Compact embedding: If $s_1<s_2$, then the embedding $X^{s_2}_{p,q}\hookrightarrow X^{s_1}_{p,q}$ is locally compact.
		\item Algebraic properties: For $s>0$, $X^s_{p,q}\cap L^\infty$ is an algebra. Moreover, $X^s_{p,q}(\mathbb{R}^d) $ is an algebra $\Longleftrightarrow$ $X^s_{p,q}(\mathbb{R}^d)\hookrightarrow L^\infty(\mathbb{R}^d)\Longleftrightarrow s>\frac{d}{p}$ (or $s\geq d$ and $p=1$ for $F^s_{p,q}$, or $s\geq \frac{d}{p}$ and $q=1$ for $B^s_{p,q}$, respectively), and there holds:
		\[ \|fg\|_{X^s_{p,q}}\leq C\big(\|f\|_{L^\infty}\|g\|_{X^s_{p,q}}+\|f\|_{X^s_{p,q}}\|g\|_{L^\infty}\big),\qquad\forall\ f,g\in X^s_{p,q}\cap L^\infty.  \]
		\item Fatou property: If the sequence $\{f_k\}_{k\in\mathbb{N}^+}$ is uniformly bounded in $X^s_{p,q}$ and converges weakly in $\mathscr{S}'$ to $f$, then $f\in X^s_{p,q}$ and $\|f\|_{X^s_{p,q}}\leq \liminf\limits_{k\to\infty}\|f_k\|_{X^s_{p,q}}$.
		\item Complex interpolation: Let $1\leq p_0,q_0\leq \infty,1\leq p_1,q_1\leq \infty,0<\theta<1$ and
		\[ \frac{1}{q}=\frac{1-\theta}{q_0}+\frac{\theta}{q_1},\quad\frac{1}{p}=\frac{1-\theta}{p_0}+\frac{\theta}{p_1},\quad  s=(1-\theta)s_0+\theta s_1. \] 
		Then we have
		\[ \|f\|_{X^{s}_{p,q}}\leq\|f\|_{X^{s_0}_{p_0,q_0}}^{1-\theta}\|f\|_{X^{s_1}_{p_1,q_1}}^\theta,\quad \forall \ f\in X^{s_0}_{p_0,q_0}\cap X^{s_1}_{p_1,q_1}.\] 
		\item Lifting property: Let 
		\[ I_{\sigma}f(x)\triangleq\mathscr{F}^{-1}[(1+|\xi|^2)^{\frac{\sigma}{2}}\mathscr{F}f(\xi)](x),\quad f\in \mathscr{S}'(\mathbb{R}^d),\quad\sigma\in\mathbb{R}. \]
		Then $I_{\sigma}$ maps $X^s_{p,q}$ isomorphically onto $X^{s-\sigma}_{p,q}$.
		\item For any $k\in\mathbb{N}$, there exists a constant $C_k$ such that the following inequality holds:
		\begin{equation*}
			C_k^{-1}\|\nabla^kf\|_{\dot{X}^s_{p,q}}\leq\|f\|_{\dot{X}^{s+k}_{p,q}}\leq C_k\|\nabla^kf\|_{\dot{X}^s_{p,q}}.
		\end{equation*}
		\end{enumerate}
	\end{lemma}
	\begin{remark}\label{re:1}
		Using Lemma \ref{lem:Triebel-Lizorkin-properties} $(vii)$, together with the definitions of the homogeneous and inhomogeneous Triebel-Lizorkin and Besov spaces, and decomposing $f$ into low and high frequencies as $f=S_1f+\tilde{f}$, one readily gets the following equivalence:
		\[ \|\nabla\tilde{f}\|_{X^{s-1}_{p,q}}\approx\|\tilde{f}\|_{X^{s}_{p,q}}, \]
		where $A\approx B$ means that there exists a constant $C>0$ such that $C^{-1}B\leq A \leq CB$.
	\end{remark}
	\begin{lemma}\label{lem:Sjguji}
		\textup{\cite{MR1419319}} Assume that $s<0$. Then we have
		\begin{enumerate}[label={(\roman*)}]
			\item 
			\[ \Big\|\big\|2^{js}|S_jf|\big\|_{l^q}\Big\|_{L^p}\leq C\|f\|_{F^s_{p,q}} \]
			for all $f\in F^s_{p,q}$ with $(p,q)\in[1,\infty)\times[1,\infty]$ or $p=q=\infty$.
			\item 
			\[ \Big\|2^{js}\big\|S_jf\big\|_{L^p}\Big\|_{l^q}\leq C\|f\|_{B^s_{p,q}} \]
			for all $f\in B^s_{p,q}$ with $1\leq p,q\leq \infty$.
		\end{enumerate}
	\end{lemma}
	
	\begin{lemma}\label{lem:Snguji}
		\textup{\cite{MR4240785,2601.10071}} Let $s\in\mathbb{R}$, $(p,q)\in[1,\infty)\times[1,\infty]$ or $p=q=\infty$. Then there exists some $C>0$, such that 
		\[ \|S_{n+1}f\|_{F^{s+l}_{p,q}}\leq C2^{nl}\|f\|_{F^s_{p,q}}, \qquad\forall\  n,l\in\mathbb{N}.\]
	\end{lemma}
	A few facts about the Nikol'skij representations are now recalled.
	\begin{lemma}\label{lem:Nik1}
		\textup{\cite{MR1376592,MR781540}} Let $\{f_k\}_{k\geq -1}$ be a sequence of functions and $0<c_1<c_2$ such that 
		\[\textup{supp}\hat{f}_{-1}\subset\{\xi\colon|\xi|\leq c_2\},\quad\textup{supp}\hat{f}_k\subset\{\xi\colon c_12^{k}\leq|\xi|\leq c_22^k\}\quad\text{with}\quad k\in\mathbb{N} . \] 
		Suppose $s\in\mathbb{R}$ and $1\leq p,q\leq\infty$, then it holds that, for some constant $C>0$,
		\begin{equation}\label{eq:equation56}
			\Big\|\sum_{k=-1}^\infty f_k\Big\|_{F^s_{p,q}}\leq C\Big\|\big\|2^{ks}f_k\big\|_{l^q}\Big\|_{L^p},\quad\text{if}\quad1\leq p<\infty,
		\end{equation}
		and
		\begin{equation}\label{eq:equation57}
			\Big\|\sum_{k=-1}^\infty f_k\Big\|_{B^s_{p,q}}\leq C\Big\|2^{ks}\big\|f_k\big\|_{L^p}\Big\|_{l^q}.
		\end{equation}
		More precisely, if the right side of the inequality in \eqref{eq:equation56} or \eqref{eq:equation57} is finite, then $\{\sum_{k=-1}^Nf_k\}_N$ converges in $\mathscr{S}'(\mathbb{R}^d)$ to a distribution $\sum_{k=-1}^\infty f_k$ satisfying this inequality. 
	\end{lemma} 
	\begin{lemma}\label{lem:Nik2}
		\textup{\cite{MR1376592,MR781540}} Let $\{f_k\}_{k\geq -1}$ be a sequence of functions and $c>0$ such that 
		\[\textup{supp}\hat{f}_{k}\subset\{\xi\colon|\xi|\leq c2^k\},\quad k\geq -1.\]  
		Suppose $s>0$ and $1\leq p,q\leq\infty$, then it holds that, for some constant $C>0$,
		\begin{equation}\label{eq:equation58}
			\Big\|\sum_{k=-1}^\infty f_k\Big\|_{F^s_{p,q}}\leq C\Big\|\big\|2^{ks}f_k\big\|_{l^q}\Big\|_{L^p},\qquad\text{if}\quad 1\leq p<\infty.
		\end{equation}
		and
		\begin{equation}\label{eq:equation59}
			\Big\|\sum_{k=-1}^\infty f_k\Big\|_{B^s_{p,q}}\leq C\Big\|2^{ks}\big\|f_k\big\|_{L^p}\Big\|_{l^q}.
		\end{equation}
		More precisely, if the right side of the inequality in \eqref{eq:equation58} or \eqref{eq:equation59} is finite, then $\{\sum_{k=-1}^Nf_k\}_N$ converges in $\mathscr{S}'(\mathbb{R}^d)$ to a distribution $\sum_{k=-1}^\infty f_k$ satisfying this inequality. 
	\end{lemma}
	
	If $f$ is a locally Lebesgue integrable function on $\mathbb{R}^d$, then the Hardy-Littlewood maximal function of $f$ is defined as follows:
	\[ (Mf)(x)=\sup_{r>0}\frac{1}{|\mathscr{B}(x,r)|}\int_{\mathscr{B}(x,r)}|f(y)|dy, \]
	where $|\mathscr{B}(x,r)|$ is the volume of the ball $\mathscr{B}(x,r)$ with center $x$ and radius $r$.
	\begin{lemma}\label{lem:max}
		\textup{\cite{MR1232192}} If $f\in L^p$, $1<p\leq\infty$, then $M(f)\in L^p$ and there exists a constant $C>0$ such that
		\[\|M(f)\|_{L^p}\leq C\|f\|_{L^p}.   \]
	\end{lemma}
	\begin{lemma}\label{lem:Fguji}
		\textup{\cite{MR284802,MR1232192}} (Fefferman-Stein vector-valued maximal inequality) Let $(p,q)\in(1,\infty)\times(1,\infty]$ or $p=q=\infty$. Suppose $\{f_j\}_{j\in\mathbb{Z}}$ is a sequence of functions in $L^p$ with the property that $||f_j(\cdot)||_{l_j^q}\in L^p(\mathbb{R}^d)$, then there exists some $C>0$, such that
		\[\Big\|\big\|Mf_j(\cdot)\big\|_{l^q_j}\Big\|_{L^p} \leq C \Big\|\big\|f_j(\cdot)\big\|_{l^q_j}\Big\|_{L^p}. \]
	\end{lemma}
	\begin{lemma}\label{lem:miaoguji}
		\textup{\cite{MR290095}} Let $\varphi$ be an integrable function on $\mathbb{R}^d $, and set $\varphi_\varepsilon(x)=\frac{1}{\varepsilon^d}\varphi(\frac{x}{\varepsilon})$ for $\varepsilon>0$. Suppose that the least decreasing radial majorant of $\varphi$ is integrable, i.e. let \[ \psi(x)=\sup_{|y|\geq|x|}|\varphi(y)| \]and we suppose $\int_{\mathbb{R}^d}\psi(x)=A<\infty$. Then for all $f\in L^p(\mathbb{R}^d),1\leq p\leq\infty, $
		\[ \sup_{\varepsilon>0}|(f\ast\varphi_\varepsilon)(x)|\leq AM(f)(x),  \]
		where $M(f)$ is the Hardy-Littlewood maximal function of $f$.
	\end{lemma}
	\begin{lemma}\label{lem:guoguji}
		\textup{\cite{MR4240785,MR781540}} Let $L>0,j,k\in\mathbb{Z},j>k-L$ and $r \in (0,\infty)$. $ \psi\in C^{\infty}(\mathbb{R}^d) $ satisfies
		\[ |\psi(z)|(1+|z|^{\frac{d}{r}})\leq g(z), \]
		where $g(z)$ is some nonnegative radial decreasing integrable function. Denote $\psi_k(x)=2^{kd}\psi(2^kx)$, then for any $\theta\in (0,1]$, there exists a constant $C$ independent of $j,k$, such that the following inequality
		\[ |(\psi_k\ast f)(x)|\leq C2^{(j-k)\theta\frac{d}{r}}M(|f|^{1-\theta})(x)[M(|f|^r)(x)]^{\frac{\theta}{r}}\]
		holds for all $f\in L_{\mathscr{B}(0,c2^j)}^p$ with $p\geq 1$ and some generic constant c, where $L_{\mathscr{B}(0,c2^j)}^p\triangleq\{f\in L^p\colon\textup{supp}\hat{f}\subset \mathscr{B}(0,c2^j)\}$. Moreover, if $j\leq k+L$, then we have
		\[ |\psi_k\ast f(x)|\leq C[M(|f|^r)(x)]^{\frac{1}{r}}. \]
	\end{lemma}
	
	Finally, we recall the following Morse type inequalities. 
	
	\begin{lemma}\label{lem:product}
		\textup{(\cite{MR1419319})} Assume $s_1\leq s_2$ and $s_1+s_2>d \max(0,\frac{2}{p}-1)$. 
		Let 
		$s_2>\frac{d}{p}$ and $q\geq\max(q_1,q_2)$. 
		In the case $s_2>s_1$, we have 
		\[  \|fg\|_{X^{s_1}_{p,q_1}(\mathbb{R}^d)}\leq \|f\|_{X^{s_1}_{p,q_1}(\mathbb{R}^d)}\|g\|_{X^{s_2}_{p,q_2}(\mathbb{R}^d)}.\]
		If $s_1=s_2$, then 
		\[ \|fg\|_{X^{s_1}_{p,q}(\mathbb{R}^d)}\leq \|f\|_{X^{s_1}_{p,q_1}(\mathbb{R}^d)}\|g\|_{X^{s_1}_{p,q_2}(\mathbb{R}^d)}.\]
	\end{lemma}
	\begin{lemma}\label{lem:product-1}
		\textup{(\cite{MR1419319})}
		Assume $s<d/p$ and 
		 $s+\frac{d}{p}>d\max(0,\frac{2}{p}-1)$. 
		Then 
		\[ \|fg\|_{X^{s}_{p,q}(\mathbb{R}^d)}\leq C \|f\|_{X^s_{p,q}(\mathbb{R}^d)}(\|g\|_{X^{d/p}_{p,\infty}(\mathbb{R}^d)}+\|g\|_{L^\infty(\mathbb{R}^d)}), \]
		holds for all $f\in X^s_{p,q}(\mathbb{R}^d)$ and $g\in X^{d/p}_{p,\infty}(\mathbb{R}^d)\cap L^\infty(\mathbb{R}^d)$.
	\end{lemma}
	\begin{remark}\label{remark:product-critical}
		 For the Triebel-Lizorkin case with $p=1,s=d-1$ and $d\geq 2$, a sharper form of Lemma \textup{\ref{lem:product-1}}  is available. Since the embedding $F^d_{1,\infty}\hookrightarrow L^\infty$ is valid, the term $\|g\|_{L^\infty}$ is redundant. Moreover, note that $\|g\|_{F^d_{1,\infty}}\leq \|g\|_{F^d_{1,q}}$. Therefore, one can further get
		\[ \|fg\|_{F^{d-1}_{1,q}(\mathbb{R}^d)}\leq C \|f\|_{F^{d-1}_{1,q}(\mathbb{R}^d)}\|g\|_{F^{d}_{1,q}(\mathbb{R}^d)},\qquad  \forall\ d\geq 2. \]
		However, the above estimate seems not to be valid for $d=1$.
	\end{remark}

	\section{Commutator estimates}
	This section is devoted to the derivation of the commutator estimates Theorem \ref{prop:commutator-estimates}.
	
	\begin{proof}[Proof of Theorem \ref{prop:commutator-estimates}]
		 Let us firstly recall Bony's para-product decomposition \cite{MR631751} as follows:
		\[ vf=T_vf+T_fv+R(v,f), \]
		where
		\[ T_vf=\sum_{k\leq j-2}\Delta_kv\Delta_jf=\sum_{j\in\mathbb{Z}}S_{j-1}v\Delta_jf, \]
		\[ R(v,f)=\sum_{j\in\mathbb{Z}}\Delta_jv\tilde{\Delta}_jf\quad\text{with}\quad\tilde{\Delta}_j\triangleq\Delta_{j-1}+\Delta_{j}+\Delta_{j+1}. \]
		In order to show that only the gradient part of $v$ is involved in the estimates, we shall split $v$ into low and high frequencies: $v=S_1v+\tilde{v}$. Obviously, there exists a constant $C$ such that 
		\[ \|S_1\nabla v\|_{X^s_{p,q}}\leq C \|\nabla v\|_{X^s_{p,q}}\quad\text{and}\quad\|\nabla \tilde{v}\|_{X^s_{p,q}}\leq C\|\nabla v\|_{X^s_{p,q}}. \]
		On the other hand, Remark \ref{re:1} ensures
		\[ \|\nabla\tilde{v}\|_{X^{s-1}_{p,q}}\approx\|\tilde{v}\|_{X^s_{p,q}}. \]
		By the Einstein convention on the summation over repeated indices $i\in\{1,\cdots,d\}$, and the Bony para-product decomposition, one can see
		\begin{align*}
			[v,\Delta_j]\cdot \nabla f
			&=[\tilde{v}^i,\Delta_j]\partial_i f+[S_0v^i,\Delta_j]\partial_i f\\
			&=[T_{\tilde{v}^i},\Delta_j]\partial_if-\Delta_j\partial_iR(\tilde{v}^i,f)+\Delta_jR(\textup{div}\tilde{v},f)+T_{\Delta_j\partial_if}\tilde{v}^i\\
			&\quad+R(\tilde{v}^i,\Delta_j\partial_if)-\Delta_jT_{\partial_if}\tilde{v}^i+[S_1v^i,\Delta_j]\partial_i f\\
			&\triangleq I_1+I_2+I_3+I_4+I_5+I_6+I_7.
		\end{align*}
		We first consider the case of Triebel-Lizorkin spaces.
		
		\textbf{Bounds for $\Big\|\big\|2^{js}|I_1|\big\|_{l^q}\Big\|_{L^p}$:}
		
		For $(p,q) \in (1,\infty)\times(1,\infty]$ or $p=q=\infty$, using \eqref{a.o.c-1}, integration by parts, first-order Taylor's formula and Lemma \ref{lem:miaoguji}, we have
		\begin{align*}
			|I_1|&=\left|\sum_{m\in\mathbb{Z}}S_{m-1}\tilde{v}^i\Delta_m\Delta_j\partial_if-\Delta_j(\sum_{m\in\mathbb{Z}}S_{m-1}\tilde{v}^i\Delta_m\partial_if)\right|\\
			&=\Bigg|\sum_{|m-j|\leq 4}S_{m-1}\tilde{v}^i(x)\Delta_m2^{jd}\int_{\mathbb{R}^d}h\left(2^j(x-y)\right)\partial_if(y)dy\\
			&\mathrel{\phantom{=}}-2^{jd}\int_{\mathbb{R}^d}h\left(2^j(x-y)\right)\left(S_{m-1}\tilde{v}^i(y)\Delta_m\partial_if(y)\right)dy\Bigg|\\
			&=\left|\sum_{|m-j|\leq 4}\int_{\mathbb{R}^d}\left(S_{m-1}\tilde{v}^i(x)-S_{m-1}\tilde{v}^i(y)\right)2^{jd}h\left(2^j(x-y)\right)\partial_i\Delta_mf(y)dy\right|\\
			&=\Bigg|\sum_{|m-j|\leq 4}\int_{\mathbb{R}^d}-\left(\partial_iS_{m-1}\tilde{v}^i(x)-\partial_iS_{m-1}\tilde{v}^i(y)\right)2^{jd}h\left(2^j(x-y)\right)\Delta_mf(y)\\
			&\mathrel{\phantom{=}}-\left(S_{m-1}\tilde{v}^i(x)-S_{m-1}\tilde{v}^i(y)\right)2^{jd+j}(\partial_ih)\left(2^j(x-y)\right)\Delta_mf(y)dy\Bigg|\\
			&\lesssim\left\|\nabla v\right\|_{L^\infty}\sum_{|m-j|\leq 4}\left|\int_{\mathbb{R}^d}2^{jd}h\left(2^j(x-y)\right)\Delta_mf(y)dy\right|\\
			&\mathrel{\phantom{=}}+\left\|\nabla S_{m-1}v\right\|_{L^\infty}\sum_{|m-j|\leq 4}\left|\int_{\mathbb{R}^d}2^{jd+j}|x-y|(\partial_ih)\left(2^j(x-y)\right)\Delta_mf(y)dy\right|\\
			&\lesssim\sum_{|m-j|\leq 4}\|\nabla v\|_{L^\infty}M(|\Delta_mf|)(x)+\sum_{|m-j|\leq 4}\|\nabla S_{m-1} v\|_{L^\infty}M(|\Delta_mf|)(x)\\
			&\lesssim\sum_{|m-j|\leq 4}\|\nabla v\|_{L^\infty}M(|\Delta_mf|)(x).	   	
		\end{align*}
		Thanks to Young's inequality and Lemma \ref{lem:Fguji}, one infers
		\begin{align}
			\Big\|\big\|2^{js}|I_1|\big\|_{l^q}\Big\|_{L^p}
			&\lesssim\|\nabla v\|_{L^\infty}\Bigg\|\bigg\|2^{js}\sum_{|m-j|\leq 4}M(|\Delta_mf|)(x)	\bigg\|_{l^q}\Bigg\|_{L^p}\nonumber\\
			&=\|\nabla v\|_{L^\infty}\Bigg\|\bigg\|\sum_{|m-j|\leq 4}2^{(j-m)s}M(2^{ms}|\Delta_mf|)(x)	\bigg\|_{l^q}\Bigg\|_{L^p}\nonumber\\
			&\lesssim\|\nabla v\|_{L^\infty}\sum_{j\in\mathbb{Z}}2^{js}\chi_{\{|j|\leq 4\}}\Big\|\big\|M(2^{ms}|\Delta_mf|)(x)	\big\|_{l^q}\Big\|_{L^p}\nonumber\\
			&\lesssim\sum_{j\in\mathbb{Z}}2^{js}\chi_{\{|j|\leq 4\}}\|\nabla v\|_{L^\infty}\Big\|\big\|2^{ms}|\Delta_mf|	\big\|_{l^q}\Big\|_{L^p}\nonumber\\
			&\lesssim\sum_{j\in\mathbb{Z}}2^{js}\chi_{\{|j|\leq 4\}}\|\nabla v\|_{L^\infty}\|f\|_{F^s_{p,q}}\label{eq:I_1-Remark}\\
			&\lesssim\|\nabla v\|_{L^\infty}\|f\|_{F^s_{p,q}}\label{eq:I_1}.
		\end{align}
		For $p=1$ with $q\in[1,\infty]$	or $q=1$ with $p\in(1,\infty)$, we first note that 
		\begin{align*}
			|I_1|&=\left|\sum_{|m-j|\leq 4}\int \left(S_{m-1}\tilde{v}^i(x)-S_{m-1}\tilde{v}^i(y)\right)
			2^{jd}h\left(2^j\left(x-y\right)\right)\Delta_m\partial_if(y)dy\right|\\
			&\lesssim\left|\sum_{|m-j|\leq 4}\int \left(S_{m-1}\tilde{v}^i(x)-S_{m-1}\tilde{v}^i(y)\right)
			2^{j(d+1)}(\partial_ih)\left(2^j\left(x-y\right)\right)\Delta_mf(y)dy\right|\\
			&\mathrel{\phantom{=}}+\left|\sum_{|m-j|\leq 4}\int \left(S_{m-1}\partial_i\tilde{v}^i(x)-S_{m-1}\partial_i\tilde{v}^i(y)\right)
			2^{jd}h\left(2^j\left(x-y\right)\right)\Delta_mf(y)dy\right|\\
			&\lesssim\left\|\nabla S_{m-1}v\right\|_{L^\infty
			}\left|\sum_{|m-j|\leq 4}\int\left|x-y\right| 
			2^{j(d+1)}\nabla h\left(2^j\left(x-y\right)\right)\Delta_mf(y)dy\right|\\
			&\mathrel{\phantom{=}}+\|\nabla v\|_{L^\infty}\left|\sum_{|m-j|\leq 4}\int 
			2^{jd} h\left(2^j\left(x-y\right)\right)\Delta_mf(y)dy\right|\\
			&\lesssim\sum_{|m-j|\leq 4}\left\|\nabla v\right\|_{L^\infty}\left[M(|\Delta_mf|^r)(x)\right]^{\frac{1}{r}
			},
		\end{align*}
		here we used first-order Taylor's formula and Lemma \ref{lem:guoguji} with $\theta=1$ and $r\in(0,1)$. Therefore, it follows from Young’s inequality and  Lemma \ref{lem:Fguji} that
		\begin{align}
			\Big\|\big\|2^{js}|I_1|\big\|_{l^q}\Big\|_{L^p} 
			&\lesssim \|\nabla v\|_{L^\infty} \Bigg\|\bigg\|\sum_{|m-j|\leq 4} 2^{(j-m)s} \left[M(|2^{ms}\Delta_m f|^r)(x)\right]^{\frac{1}{r}}\bigg\|_{l^q}\Bigg\|_{L^p} \nonumber\\
			&\lesssim \|\nabla v\|_{L^\infty}\sum_{j\in\mathbb{Z}}2^{js}\chi_{\{|j|\leq 4\}}\Big\|\big\|[M(|2^{ms}\Delta_mf|^r)(x)]^{\frac{1}{r}}\big\|_{l^q}\Big\|_{L^p}\nonumber\\
			&=\sum_{j\in\mathbb{Z}}2^{js}\chi_{\{|j|\leq 4\}}\|\nabla v\|_{L^\infty}\Big\|\big\|[M(|2^{ms}\Delta_mf|^r)(x)]\big\|_{l^{\frac{q}{r}}}\Big\|_{L^{\frac{p}{r}}}^{\frac{1}{r}}\nonumber\\
			&\lesssim\sum_{j\in\mathbb{Z}}2^{js}\chi_{\{|j|\leq 4\}}\|\nabla v\|_{L^\infty}\Big\|\big\|2^{ms}|\Delta_mf|^r(x)\big\|_{l^{\frac{q}{r}}}\Big\|_{L^{\frac{p}{r}}}^{\frac{1}{r}}\nonumber\\
			&\lesssim \sum_{j\in\mathbb{Z}}2^{js}\chi_{\{|j|\leq 4\}}\|\nabla v\|_{L^\infty}\|f\|_{F^s_{p,q}}\label{eq:endpoint-I_1-remark}\\
			&\lesssim\|\nabla v\|_{L^\infty}\|f\|_{F^s_{p,q}}\label{eq:endpoint-I_1}.
		\end{align}
		
		\textbf{Bounds for $\Big\|\big\|2^{js}|I_2|\big\|_{l^q}\Big\|_{L^p}$:}
		
		Appealing to Lemma \ref{lem:Triebel-Lizorkin-properties} $(i)$, Lemma \ref{lem:Nik2}, together with H\"{o}lder's inequality, we deduce that when $s>-1-\frac{d}{p}$ and $p\geq 2$,
		\begin{align}
			\Big\|\big\|2^{js}|I_2|\big\|_{l^q}\Big\|_{L^p}&=\Big\|\big\|2^{js}\Delta_j\partial_iR(\tilde{v}^i,f)\big\|_{l^q}\Big\|_{L^p}=\|\nabla R(\tilde{v},f)\|_{F^s_{p,q}}\leq\|R(\tilde{v},f)\|_{F^{s+1}_{p,q}}\nonumber\\
			&\lesssim \|R(\tilde{v},f)\|_{F^{s+1+\frac{d}{p}}_{\frac{p}{2},\infty}}= \Big\|\sum_{m\in\mathbb{Z}}\Delta_m\tilde{v}\tilde{\Delta}_mf\Big\|_{F^{s+1+\frac{d}{p}}_{\frac{p}{2},\infty}}\nonumber\\
			&\lesssim \Big\|\big\|2^{m(s+1+\frac{d}{p})}\Delta_m\tilde{v}\tilde{\Delta}_mf\big\|_{l^\infty}\Big\|_{L^{\frac{p}{2}}}\nonumber\\
			&\lesssim \Big\|\big\|2^{m(1+\frac{d}{p})}\Delta_m\tilde{v}\big\|_{l^\infty}\big\|2^{ms}\tilde{\Delta}_mf\big\|_{l^\infty}\Big\|_{L^{\frac{p}{2}}}\nonumber\\
			&\lesssim \Big\|\big\|2^{m(1+\frac{d}{p})}\Delta_m\tilde{v}\big\|_{l^\infty}\Big\|_{L^p}\|f\|_{F^s_{p,\infty}}\nonumber\\
			&\lesssim \|\tilde{v}\|_{F^{1+\frac{d}{p}}_{p,\infty}}\|f\|_{F^s_{p,\infty}}\lesssim \|\nabla\tilde{v}\|_{F^{\frac{d}{p}}_{p,\infty}}\|f\|_{F^s_{p,q}}\nonumber\\
			&\lesssim \|\nabla v\|_{F^{\frac{d}{p}}_{p,\infty}}\|f\|_{F^s_{p,q}}\label{eq:I_2}.
		\end{align}
		Now, if $1\leq p<2$, by virtue of the relation $\frac{1}{p}+\frac{1}{p'}=1$ and an analogous estimate to the one above, we obtain that when $s>-1-\frac{d}{p'}$,
		\begin{align}
			\Big\|\big\|2^{js}|I_2|\big\|_{l^q}\Big\|_{L^p}&=\Big\|\big\|2^{js}\Delta_j\partial_iR(\tilde{v}^i,f)\big\|_{l^q}\Big\|_{L^p}=\|\nabla R(\tilde{v},f)\|_{F^s_{p,q}}\leq\|R(\tilde{v},f)\|_{F^{s+1}_{p,q}}\nonumber\\
			&\lesssim \|R(\tilde{v},f)\|_{F^{s+1+\frac{d}{p'}}_{1,\infty}}= \Big\|\sum_{m\in\mathbb{Z}}\Delta_m\tilde{v}\tilde{\Delta}_mf\Big\|_{F^{s+1+\frac{d}{p'}}_{1,\infty}}\nonumber\\
			&\lesssim \Big\|\big\|2^{m(s+1+\frac{d}{p'})}\Delta_m\tilde{v}\tilde{\Delta}_mf\big\|_{l^\infty}\Big\|_{L^1}\nonumber\\
			&\lesssim \Big\|\big\|2^{m(1+\frac{d}{p'})}\Delta_m\tilde{v}\big\|_{l^\infty}\big\|2^{ms}\tilde{\Delta}_mf\big\|_{l^\infty}\Big\|_{L^{1}}\nonumber\\
			&\lesssim \Big\|\big\|2^{m(1+\frac{d}{p'})}\Delta_m\tilde{v}\big\|_{l^\infty}\Big\|_{L^{p'}}\|f\|_{F^s_{p,\infty}}\nonumber\\
			&\lesssim \|\tilde{v}\|_{F^{1+\frac{d}{p'}}_{p',\infty}}\|f\|_{F^s_{p,\infty}}\lesssim \|\tilde{v}\|_{F^{1+\frac{d}{p}}_{p,\infty}}\|f\|_{F^s_{p,q}}\nonumber\\
			&\lesssim \|\nabla v\|_{F^{\frac{d}{p}}_{p,\infty}}\|f\|_{F^s_{p,q}}\label{eq:I_2'}.
		\end{align}
		
		\textbf{Bounds for $\Big\|\big\|2^{js}|I_3|\big\|_{l^q}\Big\|_{L^p}$:}
		
		Similar to \eqref{eq:I_2}, we obtain that for $s>-\frac{d}{p}$ and $p\geq 2$,
		\begin{align}
			\Big\|\big\|2^{js}|I_3|\big\|_{l^q}\Big\|_{L^p}&=\Big\|\big\|2^{js}\Delta_jR(\text{div}\tilde{v},f)\big\|_{l^q}\Big\|_{L^p}=\| R(\text{div}\tilde{v},f)\|_{F^s_{p,q}}\lesssim \|R(\text{div}\tilde{v},f)\|_{F^{s+\frac{d}{p}}_{\frac{p}{2},\infty}}\nonumber\\
			&\lesssim \Big\|\big\|2^{m(s+\frac{d}{p})}\Delta_m\text{div}\tilde{v}\tilde{\Delta}_mf\big\|_{l^\infty}\Big\|_{L^{\frac{p}{2}}}\nonumber\\
			&\lesssim\|\nabla v\|_{F^{\frac{d}{p}}_{p,\infty}}\|f\|_{F^s_{p,q}}\label{eq:I_3}.
		\end{align}
		If $1\leq p<2$, by an estimate analogous to \eqref{eq:I_2'}, we have for $s>-\frac{d}{p'}$,
		\begin{align}
			\Big\|\big\|2^{js}|I_3|\big\|_{l^q}\Big\|_{L^p}&=\Big\|\big\|2^{js}\Delta_jR(\text{div}\tilde{v},f)\big\|_{l^q}\Big\|_{L^p}=\| R(\text{div}\tilde{v},f)\|_{F^s_{p,q}}\lesssim \|R(\text{div}\tilde{v},f)\|_{F^{s+\frac{d}{p'}}_{1,\infty}}\nonumber\\
			&\lesssim \Big\|\big\|2^{m(s+\frac{d}{p'})}\Delta_m\text{div}\tilde{v}\tilde{\Delta}_mf\big\|_{l^\infty}\Big\|_{L^{1}}\nonumber\\
			&\lesssim\|\nabla v\|_{F^{\frac{d}{p}}_{p,\infty}}\|f\|_{F^s_{p,q}}\label{eq:I_3'}.
		\end{align}
		
		\textbf{Bounds for $\Big\|\big\|2^{js}|I_4+I_5|\big\|_{l^q}\Big\|_{L^p}$:}
		
		Note that $S_{m+2}\Delta_jf=0$ if $m\leq j-3$. Then, applying Lemma \ref{lem:Bernstein}, we obtain
		\begin{align}
			|I_4+I_5|&=\big|T_{\partial_i\Delta_jf}\tilde{v}^i+R(\tilde{v}^i,\Delta_j\partial_if)\big|=\Big|\sum_{m\geq j-2}S_{m+2}\partial_i\Delta_jf\Delta_m\tilde{v}^i\Big|\nonumber\\
			&\leq\Big|\sum_{m\geq j-2} \|\nabla \Delta_m\tilde{v}\|_{L^\infty}2^{-m}S_{m+2}\partial_i\Delta_jf\Big|\nonumber\\
			&\leq \|\nabla v\|_{L^\infty}\Big|\sum_{m\geq j-2} 2^{-m}S_{m+2}\partial_i\Delta_jf\Big|\nonumber\\
			&=\|\nabla v\|_{L^\infty}\Big|\sum_{m\geq j-2}\int_{\mathbb{R}^d}\big(2^{-m}\cdot2^{(m+2)d}\tilde{h}(2^{m+2}(x-y))\partial_i\Delta_jf(y)\big)dy\Big|\label{eq:equation60}.
		\end{align}
		For $(p,q) \in (1,\infty)\times(1,\infty]$ or $p=q=\infty$, an application of Lemma \ref{lem:miaoguji} to \eqref{eq:equation60}, followed by Young's inequality and Lemma \ref{lem:Fguji}, yields
		\begin{align}
			\Big\|\big\|2^{js}|I_4+I_5|\big\|_{l^q}\Big\|_{L^p}&\lesssim\|\nabla v\|_{L^\infty}\Big\|\big\|\sum_{m\geq j-2}2^{js-m}M(\partial_i\Delta_jf)\big\|_{l^q}\Big\|_{L^p}\nonumber\\
			&=\|\nabla v\|_{L^\infty}\Big\|\big\|\sum_{m\geq j-2}2^{j(s-1)}\cdot2^{j-m}M(\partial_i\Delta_jf)\big\|_{l^q}\Big\|_{L^p}\nonumber\\
			&\lesssim\left\|\nabla v\right\|_{L^\infty}\sum_{j\in\mathbb{Z}}2^{j}\chi_{\{j\leq 2\}}\Big\|\big\|2^{j(s-1)}M(\partial_i\Delta_jf) \big\|_{l^q}\Big\|_{L^p}\nonumber\\
			&\lesssim\sum_{j\in\mathbb{Z}}2^{j}\chi_{\{j\leq 2\}}\left\|\nabla v\right\|_{L^\infty}\|\nabla f\|_{F^{s-1}_{p,q}}\label{eq:I_4+I_5-remark}\\
			&\lesssim\left\|\nabla v\right\|_{L^\infty}\|f\|_{F^s_{p,q}}\label{eq:I_4+I_5}.
		\end{align}
		
		In the case $p=1$ with $q\in[1,\infty]$	or $q=1$ with $p\in(1,\infty)$, we apply Lemma \ref{lem:guoguji} to \eqref{eq:equation60} and, following an analogous argument, we get
		\begin{align}
			\Big\|\big\|2^{js}|I_4+I_5|\big\|_{l^q}\Big\|_{L^p}&\lesssim\|\nabla v\|_{L^\infty}\Big\|\big\|\sum_{m\geq j-2}2^{js-m}[M(|\partial_i\Delta_jf|^r)]^{\frac{1}{r}}\big\|_{l^q}\Big\|_{L^p}\nonumber\\
			&\lesssim\sum_{j\in\mathbb{Z}}2^{j}\chi_{\{j\leq 2\}}\left\|\nabla v\right\|_{L^\infty}\|\nabla f\|_{F^{s-1}_{p,q}}\label{eq:endpoint-I_4+I_5-remark}\\
			&\lesssim\left\|\nabla v\right\|_{L^\infty}\|f\|_{F^s_{p,q}}\label{eq:endpoint-I_4+I_5}.
		\end{align}
		
		\textbf{Bounds for $\Big\|\big\|2^{js}|I_6|\big\|_{l^q}\Big\|_{L^p}$:}
		
		By Lemma \ref{lem:Nik1}, Hölder's inequality (where $\frac{1}{p}=\frac{1}{p_1}+\frac{1}{p_2}$), Lemma \ref{lem:Sjguji}, and Lemma \ref{lem:Triebel-Lizorkin-properties} $(i)$, it follows that for $s<1+\frac{d}{p}$,
		\begin{align}
			\Big\|\big\|2^{js}|I_6|\big\|_{l^q}\Big\|_{L^p}&=\Big\|\big\|2^{js}\sum_{|m-j|\leq 4}\Delta_j(S_{m-1}\partial_i f\Delta_m\tilde{v}^i)\big\|_{l^q}\Big\|_{L^p}\nonumber\\
			&\lesssim \Big\|\big\|2^{ms}S_{m-1}\partial_if\Delta_m\tilde{v}^i\big\|_{l^q}\Big\|_{L^p}\nonumber\\
			&\lesssim \Big\|\big\|2^{m(s-1-\frac{d}{p_2})}\cdot2^{m(1+\frac{d}{p_2})}S_{m-1}\partial_if\Delta_m\tilde{v}^i\big\|_{l^q}\Big\|_{L^p}\nonumber\\
			&\lesssim \Big\|\big\|2^{m(s-1-\frac{d}{p_2})}S_{m-1}\partial_if\big\|_{l^q}\big\|2^{m(1+\frac{d}{p_2})}\Delta_m\tilde{v}^i\big\|_{l^\infty}\Big\|_{L^p}\nonumber\\
			&\lesssim\|\nabla f\|_{F^{s-1-\frac{d}{p_2}}_{p_1,q}}\|\tilde{v}\|_{F^{1+\frac{d}{p_2}}_{p_2,\infty}}\nonumber\\
			&\lesssim\|\nabla \tilde{v}\|_{F^{\frac{d}{p_2}}_{p_2,\infty}}\|\nabla f\|_{F^{s-1}_{p,q}}\nonumber\\
			&\lesssim\|\nabla v\|_{F^{\frac{d}{p}}_{p,\infty}}\| f\|_{F^{s}_{p,q}}\label{eq:I_6-sxiaoyu}.
		\end{align}
		On the other hand, we observe that
		\begin{align}
			|I_6|&=\Big|\sum_{|m-j|\leq4}\Delta_j(S_{m-1}\partial_if\Delta_m\tilde{v}^i)\Big|\nonumber\\
			&=\Big|\sum_{|m-j|\leq4}\int_{\mathbb{R}^d}\big(2^{jd}h(2^j(x-y))(S_{m-1}\partial_if\Delta_m\tilde{v}^i)(y)\big)dy\Big|.
		\end{align}
		For $(p,q) \in (1,\infty)\times(1,\infty]$ or $p=q=\infty$, using Lemma \ref{lem:miaoguji}, Young's inequality, and Lemma \ref{lem:Fguji}, we obtain
		\begin{align}
			\Big\|\big\|2^{js}|I_6|\big\|_{l^q}\Big\|_{L^p}&\lesssim\|\nabla f\|_{L^\infty}\Big\|\big\|2^{js}\sum_{|m-j|\leq 4}M(\Delta_m\tilde{v}^i)\big\|_{l^q}\Big\|_{L^p}\nonumber\\
			&\lesssim\|\nabla f\|_{L^\infty}\Big\|\big\|\sum_{|m-j|\leq 4}2^{(j-m)s}\cdot2^{ms}M(\Delta_m\tilde{v}^i)\big\|_{l^q}\Big\|_{L^p}\nonumber\\
			&\lesssim\|\nabla f\|_{L^\infty} \sum_{j\in\mathbb{Z}}2^{js}\chi_{\{|j|\leq 4\}}\Big\|\big\|2^{ms}M(\Delta_m\tilde{v}^i)\big\|_{l^q}\Big\|_{L^p}\nonumber\\
			&\lesssim\sum_{j\in\mathbb{Z}}2^{js}\chi_{\{|j|\leq 4\}}\|\nabla f\|_{L^\infty}\|\tilde{v}\|_{F^s_{p,q}}\label{eq:I_6-remark}\\
			&\lesssim\|\nabla f\|_{L^\infty}\|\nabla \tilde{v}\|_{F^{s-1}_{p,q}}\lesssim\|\nabla f\|_{L^\infty}\|\nabla v\|_{F^{s-1}_{p,q}}.\label{eq:I_6}
		\end{align}
		Similarly, for $p=1$ with $q\in[1,\infty]$	or $q=1$ with $p\in(1,\infty)$, using Lemma \ref{lem:guoguji}, Young's inequality, and Lemma \ref{lem:Fguji}, one can get
		\begin{align}
			\Big\|\big\|2^{js}|I_6|\big\|_{l^q}\Big\|_{L^p}&\lesssim\|\nabla f\|_{L^\infty}\Big\|\big\|2^{js}\sum_{|m-j|\leq 4}[M(|\Delta_m\tilde{v}^i|^r)]^{\frac{1}{r}}\big\|_{l^q}\Big\|_{L^p}\nonumber\\
			&\lesssim\sum_{j\in\mathbb{Z}}2^{js}\chi_{\{|j|\leq 4\}}\|\nabla f\|_{L^\infty}\|\nabla \tilde{v}\|_{F^{s-1}_{p,q}}\label{eq:endpoint-I_6-remark}\\
			&\lesssim\|\nabla f\|_{L^\infty}\|\nabla v\|_{F^{s-1}_{p,q}}\label{eq:endpoint-I_6}.
		\end{align}
		When $s>1+\frac{d}{p}$, it is easy to deduce that
		\begin{equation}\label{eq:I_6sdayu}
			\Big\|\big\|2^{js}|I_6|\big\|_{l^q}\Big\|_{L^p}\lesssim\|f\|_{F^s_{p,q}}\|\nabla v\|_{F^{s-1}_{p,q}}.
		\end{equation}
		Using Bernstein's inequality, we can also derive an estimate of the following form. For brevity, we only treat the cases $(p,q)=(1,\infty)\times(1,\infty]$ or $p=q=\infty$, for the remaining cases with $p=1$ with $q\in[1,\infty]$ and $q=1$ with $p\in(1,\infty)$, we refer to \eqref{eq:endpoint-I_6}.
		\begin{align}
			\Big\|\big\|2^{js}|I_6|\big\|_{l^q}\Big\|_{L^p}
			&\lesssim \| S_{m-1}f\|_{L^\infty}\Big\|\big\| \sum_{|m-j|\leq 4}2^{(j-m)s}2^{m(s+1)}M(\Delta_m\tilde{v}^i)(x)\big\|_{l^q}\Big\|_{L^p}\nonumber\\
			&\lesssim \| S_{m-1}f\|_{L^\infty}\sum_{j\in\mathbb{Z}}2^{js}\chi_{\{|j|\leq 4\}}\Big\|\big\| 2^{m(s+1)}M(\Delta_m\tilde{v}^i)(x)\big\|_{l^q}\Big\|_{L^p}\nonumber\\
			&\lesssim\sum_{j\in\mathbb{Z}}2^{js}\chi_{\{|j|\leq 4\}}\| f\|_{L^\infty}\|\tilde{v}\|_{F^{s+1}_{p,q}}\nonumber\\
			&\lesssim\|f\|_{L^\infty}\|\nabla\tilde{v}\|_{F^{s}_{p,q}}\lesssim\|f\|_{L^\infty}\|\nabla v\|_{F^{s}_{p,q}}.\label{eq:I_6-blow-up}
		\end{align}
		
		\textbf{Bounds for $\Big\|\big\|2^{js}|I_7|\big\|_{l^q}\Big\|_{L^p}$:}
		
		As $I_7=\sum_{|j-m|\leq1}[S_1v,\Delta_j]\cdot\nabla\Delta_mf  $, by the first-order Taylor formula, we have
		\begin{align}
			|I_7|&=\big|\sum_{|j-m|\leq1}(S_1v\Delta_j\nabla\Delta_mf-\Delta_j(S_1v\nabla\Delta_mf))\big|\nonumber\\
			&=\Big|\sum_{|j-m|\leq1}\int_{\mathbb{R}^d}2^{jd}(S_1v(x)-S_1v(y))h(2^j(x-y))\nabla\Delta_mf(y)dy\Big|\nonumber\\
			&=\Big|\sum_{|j-m|\leq1}\int_{\mathbb{R}^d}2^{jd}(S_1\nabla v(x)-S_1\nabla v(y))h(2^j(x-y))\Delta_mf(y)dy\Big|\nonumber\\
			&\quad+\Big|\sum_{|j-m|\leq1}\int_{\mathbb{R}^d}2^{jd+j}(S_1v(x)-S_1v(y))\nabla h(2^j(x-y))\Delta_mf(y)dy\Big|\nonumber\\
			&\leq \|\nabla v\|_{L^\infty}\Big|\sum_{|j-m|\leq1}\int_{\mathbb{R}^d}2^{jd}h(2^j(x-y))\Delta_mf(y)dy\Big|\nonumber\\
			&\quad+\|\nabla v\|_{L^\infty}\Big|\sum_{|j-m|\leq1}\int_{\mathbb{R}^d}2^{jd+j}|x-y|\partial_ih(2^j(x-y))\Delta_mf(y)dy\Big|.\label{eq:equation61}
		\end{align}	 
		Similar to the case of $I_1$,  for $(p,q) \in (1,\infty)\times(1,\infty]$ or $p=q=\infty$, applying Lemma \ref{lem:Fguji} to \eqref{eq:equation61} and then using Young's inequality together with Lemma \ref{lem:miaoguji}, we obtain 
		\begin{equation}\label{eq:I_7}
			\Big\|\big\|2^{js}|I_7|\big\|_{l^q}\Big\|_{L^p}\lesssim \|\nabla v\|_{L^\infty}\|f\|_{F^s_{p,q}}.
		\end{equation}
		For $p=1$ with $q\in[1,\infty]$	or $q=1$ with $p\in(1,\infty)$, we apply Lemma \ref{lem:Fguji} to \eqref{eq:equation61}, followed by Young's inequality and Lemma \ref{lem:guoguji}, which yields
		\begin{equation}\label{eq:endpoint-I_7}
			\Big\|\big\|2^{js}|I_7|\big\|_{l^q}\Big\|_{L^p}\lesssim   \|\nabla v\|_{L^\infty}\|f\|_{F^s_{p,q}}.
		\end{equation}
		
		Hence, combining inequalities \eqref{eq:I_1}, \eqref{eq:endpoint-I_1}, \eqref{eq:I_2}, \eqref{eq:I_2'}, \eqref{eq:I_3}, \eqref{eq:I_3'}, \eqref{eq:I_4+I_5}, \eqref{eq:endpoint-I_4+I_5},  \eqref{eq:I_6-sxiaoyu} or \eqref{eq:I_6sdayu}, \eqref{eq:I_7}, \eqref{eq:endpoint-I_7} yields \eqref{eq:commutator-estimates1} and \eqref{eq:commutator-estimates2}. 
		
		To arrive at \eqref{eq:commutator-estimates4}, \eqref{eq:commutator-estimates5} and \eqref{eq:commutator-estimates3}, we reconsider the estimate of $I_2+I_3\triangleq -\Delta_jR(\tilde{v}^i,\partial_if)$. When $(p,q) \in (1,\infty)\times(1,\infty]$ or $p=q=\infty$, by applying integration by parts, the first-order Taylor formula, and invoking Lemma \ref{lem:miaoguji}, one infers
		\begin{align}
			|I_2+I_3|&=\left|\sum_{m\in\mathbb{Z}}\Delta_j\left(\Delta_m\tilde{v}^i\tilde{\Delta}_m\partial_if\right)\right|=\left|\sum_{m\geq j-3}\Delta_j\left(\Delta_m\tilde{v}^i\tilde{\Delta}_m\partial_if\right)\right|\nonumber\\
			&=\left|\sum_{m\geq j-3}2^{jd}\int_{\mathbb{R}^d}h\left(2^j(x-y)\right)\left(\Delta_m\tilde{v}^i(y)\tilde{\Delta}_m\partial_if(y)\right)dy\right|\nonumber\\
			&\lesssim\left|\sum_{m\geq j-3}2^{jd+j}\int_{\mathbb{R}^d}(\partial_ih)\left(2^j(x-y)\right)\Delta_m\tilde{v}^i(y)\tilde{\Delta}_mf(y)dy\right|\nonumber\\
			&\mathrel{\phantom{=}}+\left|\sum_{m\geq j-3}2^{jd}\int_{\mathbb{R}^d}h\left(2^j(x-y)\right)\partial_i\Delta_m\tilde{v}^i(y)\tilde{\Delta}_mf(y)dy\right|\label{eq:I_2+I_3fenbujifenhou}\\
			&\lesssim\sum_{m\geq j-3}2^jM(\tilde{\Delta}_mf)(x)\|\Delta_m v\|_{L^\infty}+\sum_{m\geq j-3}M(\tilde{\Delta}_mf)(x)\|\nabla\Delta_mv\|_{L^\infty}.\nonumber
		\end{align}
		Now applying Lemma \ref{lem:Bernstein}, Young's inequality and Lemma \ref{lem:Fguji}, for $s>0$, one can find
		\begin{align}
			\Big\|\big\|2^{js}|I_2+I_3|\big\|_{l^q}\Big\|_{L^p}
			&\lesssim\Bigg\|\bigg\|\sum_{m\geq j-3}2^{js+j}M(\tilde{\Delta}_mf)(x)\|\Delta_m v\|_{L^\infty}\bigg\|_{l^q}\Bigg\|_{L^p}\nonumber\\
			&\mathrel{\phantom{=}}+\Bigg\|\bigg\|\sum_{m\geq j-3}2^{js}M(\tilde{\Delta}_mf)(x)\|\nabla\Delta_mv\|_{L^\infty}\bigg\|_{l^q}\Bigg\|_{L^p}\nonumber\\
			&\lesssim\Bigg\|\bigg\|\sum_{m\geq j-3}2^{(j-m)(s+1)}M(2^{ms}\tilde{\Delta}_mf)(x)\|\nabla\Delta_mv\|_{L^\infty}\bigg\|_{l^q}\Bigg\|_{L^p}\nonumber\\
			&\mathrel{\phantom{=}}+\Bigg\|\bigg\|\sum_{m\geq j-3}2^{(j-m)s}M(2^{ms}\tilde{\Delta}_mf)(x)\|\nabla\Delta_mv\|_{L^\infty}\bigg\|_{l^q}\Bigg\|_{L^p}\nonumber\\
			&\lesssim\sum_{j\in\mathbb{Z}}2^{j(s+1)}\chi_{\{j\leq 3\}}\Big\|\big\|M(2^{ms}\tilde{\Delta}_mf)(x)\|\nabla\Delta_mv\|_{L^\infty}\big\|_{l^q}\Big\|_{L^p}\nonumber\\
			&\mathrel{\phantom{=}}+\sum_{j\in\mathbb{Z}}2^{js}\chi_{\{j\leq 3\}}\Big\|\big\|M(2^{ms}\tilde{\Delta}_mf)(x)\|\nabla\Delta_mv\|_{L^\infty}\big\|_{l^q}\Big\|_{L^p}\nonumber\\
			&\lesssim\sum_{j\in\mathbb{Z}}2^{js}\chi_{\{j\leq 3\}}\|\nabla v\|_{L^\infty}\Big\|\big\|M(2^{ms}\tilde{\Delta}_mf)(x)\big\|_{l^q}\Big\|_{L^p}\nonumber\\
			&\lesssim\sum_{j\in\mathbb{Z}}2^{js}\chi_{\{j\leq 3\}}\|\nabla v\|_{L^\infty}\Big\|\big\|2^{ms}|\tilde{\Delta}_mf|\big\|_{l^q}\Big\|_{L^p}\nonumber\\
			&\lesssim\sum_{j\in\mathbb{Z}}2^{js}\chi_{\{j\leq 3\}}\|\nabla v\|_{L^\infty}\|f\|_{F^s_{p,q}}\label{eq:I_2+I_3-remark}\\
			&\lesssim\|\nabla v\|_{L^\infty}\|f\|_{F^s_{p,q}}\label{eq:I_2+I_3}.
		\end{align}
		We note that if $\text{div}v=0 $, then the second term in \eqref{eq:I_2+I_3fenbujifenhou}, which results from integration by parts, is identically zero. It then follows from the above argument that for $s>-1$ and $\text{div }v=0$,
		\begin{align}
			\Big\|\big\|2^{js}|I_2+I_3|\big\|_{l^q}\Big\|_{L^p}&\lesssim\sum_{j\in\mathbb{Z}}2^{j(s+1)}\chi_{\{j\leq 3\}}\|\nabla v\|_{L^\infty}\|f\|_{F^s_{p,q}}\label{eq:I_2+I_3-v=f-remark}\\
			&\lesssim\|\nabla v\|_{L^\infty}\|f\|_{F^s_{p,q}}\label{eq:I_2+I_3-v=f}.
		\end{align}
		For $p=1$ with $q\in[1,\infty]$	or $q=1$ with $p\in(1,\infty)$, since $s>0$, for arbitrary $r\in (0,1)$, one can select $\theta \in (0,1)$ small enough, such that $s>d\frac{\theta}{r}$. Due to frequency interaction \eqref{a.o.c}, one observes that $\Delta_j(\Delta_m\tilde{v}\tilde{\Delta}_mf)\equiv0$, if $m <j-3$. Using  integration by parts and Lemma \ref{lem:guoguji}, we deduce
		\begin{align*}
			|I_2+I_3|
			&=\left|\Delta_j\sum_{m\in\mathbb{Z}}\Delta_m\tilde{v}^i\tilde{\Delta}_m\partial_if\right|=\left|\sum_{m\geq j-3}\Delta_j\left(\Delta_m\tilde{v}^i\tilde{\Delta}_m\partial_if\right)\right|\\
			&=\left|\sum_{m\geq j-3}\int2^{jd}h\left(2^j(x-y)\right)\Delta_m\tilde{v}^i(y)\tilde{\Delta}_m\partial_if(y)dy\right|\\
			&\lesssim\left|\sum_{m\geq j-3}\int2^{jd+j}(\partial_ih)\left(2^j(x-y)\right)\Delta_m\tilde{v}^i(y)\tilde{\Delta}_mf(y)dy\right|\\
			&\mathrel{\phantom{=}}+\left|\sum_{m\geq j-3}\int2^{jd}h\left(2^j(x-y)\right)\Delta_m\partial_i\tilde{v}^i(y)\tilde{\Delta}_mf(y)dy\right|\\
			&\lesssim\sum_{m\geq j-3}2^j2^{(m-j)\frac{d\theta}{r}}M(|\Delta_m\tilde{v}^i\tilde{\Delta}_mf|^{1-\theta})\left[M(|\Delta_m\tilde{v}^i\tilde{\Delta}_mf|^r)(x)\right]^{\frac{\theta}{r}}\\
			&\mathrel{\phantom{=}}+\sum_{m\geq j-3}2^{(m-j)\frac{d\theta}{r}}M(|\Delta_m\partial_i\tilde{v}^i\tilde{\Delta}_mf|^{1-\theta})\left[M(|\Delta_m\partial_i\tilde{v}^i\tilde{\Delta}_mf|^r)(x)\right]^{\frac{\theta}{r}}.
		\end{align*}    	
		Thanks to Young's inequality, H\"{o}lder's inequality and Lemma \ref{lem:Bernstein}, one can infer
		\begin{align}
			&\mathrel{\phantom{=}}\Big\|\big\|2^{js}|I_2+I_3|\big\|_{l^q}\Big\|_{L^p}\nonumber\\
			&\lesssim\Bigg\|\bigg\|\sum_{m\geq j-3}2^{(j-m)(s+1-\frac{d\theta}{r})}M(|(2^m\Delta_m\tilde{v}^i)(2^{ms}\tilde{\Delta}_mf)|^{1-\theta})\left[M(|(2^m\Delta_m\tilde{v}^i)(2^{ms}\tilde{\Delta}_mf)|^r)(x)\right]^{\frac{\theta}{r}}\bigg\|_{l^q}\Bigg\|_{L^p}\nonumber\\
			&\mathrel{\phantom{=}}+\Bigg\|\bigg\|\sum_{m\geq j-3}2^{(j-m)(s-\frac{d\theta}{r})}M(|(\Delta_m\partial_i\tilde{v}^i)(2^{ms}\tilde{\Delta}_mf)|^{1-\theta})\left[M(|(\Delta_m\partial_i\tilde{v}^i)(2^{ms}\tilde{\Delta}_mf)|^r)(x)\right]^{\frac{\theta}{r}}\bigg\|_{l^q}\Bigg\|_{L^p}\nonumber\\
			&\lesssim\sum_{j\in\mathbb{Z}}2^{j(s+1-\frac{d\theta}{r})}\chi_{\{j\leq 3\}}\Bigg\|\bigg\|M(|(2^m\Delta_m\tilde{v}^i)(2^{ms}\tilde{\Delta}_mf)|^{1-\theta})\left[M(|(2^m\Delta_m\tilde{v}^i)(2^{ms}\tilde{\Delta}_mf)|^r)(x)\right]^{\frac{\theta}{r}}\bigg\|_{l^q}\Bigg\|_{L^p}\nonumber\\
			&\mathrel{\phantom{=}}+\sum_{j\in\mathbb{Z}}2^{j(s-\frac{d\theta}{r})}\chi_{\{j\leq 3\}}\Bigg\|\bigg\|M(|(\Delta_m\partial_i\tilde{v}^i)(2^{ms}\tilde{\Delta}_mf)|^{1-\theta})\left[M(|(\Delta_m\partial_i\tilde{v}^i)(2^{ms}\tilde{\Delta}_mf)|^r)(x)\right]^{\frac{\theta}{r}}\bigg\|_{l^q}\Bigg\|_{L^p}\nonumber\\
			&\lesssim\sum_{j\in\mathbb{Z}}2^{j(s-\frac{d\theta}{r})}\chi_{\{j\leq 3\}}\|\nabla v\|_{L^\infty}\Big\|\big\|M(|2^{ms}\tilde{\Delta}_mf|^{1-\theta})(x)\big[M(2^{ms}\tilde{\Delta}_mf)^r(x)\big]^{\frac{\theta}{r}}\big\|_{l^q}\Big\|_{L^p}\nonumber\\
			&\lesssim \sum_{j\in\mathbb{Z}}2^{j(s-\frac{d\theta}{r})}\chi_{\{j\leq 3\}}\|\nabla v\|_{L^\infty}\|f\|_{F^s_{p,q}}\label{eq:endpoint-I_2+I_3+remark}\\
			&\lesssim \|\nabla v\|_{L^\infty}\|f\|_{F^s_{p,q}}\label{eq:endpoint-I_2+I_3}.
		\end{align}
		Analogous to \eqref{eq:I_2+I_3-v=f-remark} and \eqref{eq:I_2+I_3-v=f}, we find that if $\text{div}v=0$, then for any $s>-1$, we have
		\begin{align}
			\Big\|\big\|2^{js}|I_2+I_3|\big\|_{l^q}\Big\|_{L^p}&\lesssim\sum_{j\in\mathbb{Z}}2^{j(s+1-\frac{d\theta}{r})}\chi_{\{j\leq 3\}}\|\nabla v\|_{L^\infty}\|f\|_{F^s_{p,q}}\label{eq:endpoint-I_2+I_3-v=f-remark}\\
			&\lesssim\|\nabla v\|_{L^\infty}\|f\|_{F^s_{p,q}}\label{eq:endpoint-I_2+I_3-v=f}.
		\end{align}
		Thus, combining \eqref{eq:I_1}, \eqref{eq:endpoint-I_1},  \eqref{eq:I_4+I_5}, \eqref{eq:endpoint-I_4+I_5}, \eqref{eq:I_6}, \eqref{eq:endpoint-I_6}, \eqref{eq:I_6-blow-up}, \eqref{eq:I_7}, \eqref{eq:endpoint-I_7} \eqref{eq:I_2+I_3}, \eqref{eq:I_2+I_3-v=f}, \eqref{eq:endpoint-I_2+I_3}, \eqref{eq:endpoint-I_2+I_3-v=f} yields \eqref{eq:commutator-estimates4}, \eqref{eq:commutator-estimates5} and \eqref{eq:commutator-estimates3}, thereby completing the proof of Theprem \ref{prop:commutator-estimates} in the Triebel-Lizorkin spaces. For the case of Besov spaces, by a similar argument as in the Triebel-Lizorkin setting, with Lemma \ref{lem:Fguji} replaced by Lemma \ref{lem:max}, the corresponding results can be readily obtained. For the sake of simplicity, we omit the details here. Therefore, we complete the proof of Theorem \ref{prop:commutator-estimates}. 
	\end{proof}
	
	\begin{remark}\label{remark:jiaohaunziquyu0}
		By Lemmas \ref{lem:Nik1}-\ref{lem:Nik2}, together with \eqref{eq:I_1-Remark}, \eqref{eq:endpoint-I_1-remark}, \eqref{eq:I_4+I_5-remark}, \eqref{eq:endpoint-I_4+I_5-remark}, \eqref{eq:I_6-remark}, \eqref{eq:endpoint-I_6-remark}, \eqref{eq:I_2+I_3-remark}, \eqref{eq:I_2+I_3-v=f-remark}, \eqref{eq:endpoint-I_2+I_3+remark}, \eqref{eq:endpoint-I_2+I_3-v=f-remark}, 
		it follows directly that, under the hypotheses of Theorem \ref{prop:commutator-estimates}, we have 
		\begin{flalign*}
			&&\Big\|\big\|2^{js}([f,\Delta_j]\cdot\nabla g)\big\|_{l^q(j\geq n)}\Big\|_{L^p}\to 0,\qquad \text{as}\quad n\to\infty,&&\\
			\text{and}&&\Big\|2^{js}\big\|[f,\Delta_j]\cdot\nabla g\big\|_{L^p}\Big\|_{l^q(j\geq n)}\to 0,\qquad\text{as}\quad n\to\infty,&& 
		\end{flalign*}
		where $\|f_j\|_{l^q(j\geq n)}$ stand for $\big(\sum_{j\geq n}|f_j|^q\big)^{\frac{1}{q}}$ with usual modification if $q=\infty$.
	\end{remark}

	\section{A priori estimates for transport equations}
		In this section, based on Theorem \ref{prop:commutator-estimates}, we establish the a priori estimates Theorem \ref{Thm:a-priori-estimates} for the transport equations \eqref{Eq}. 
	
	\begin{proof}[Proof of Theorem \ref{Thm:a-priori-estimates}]
		We first prove the case of Triebel-Lizorkin spaces. For this, applying the frequency localization operator $\Delta_j$ to \eqref{Eq}, one has
		\begin{equation}\label{zuoyongDelta}
			\begin{cases}
				(\partial_t+v\cdot\nabla)\Delta_jf=\Delta_j g+[v,\Delta_j]\cdot\nabla f,\\
				\Delta_jf|_{t=0}=\Delta_jf_0.
			\end{cases}
		\end{equation}
		Let us introduce particle trajectory mapping $X(t,\alpha)$, by definition, the solution to the following ordinary differential equation:
		\begin{equation}\label{liziguidaoyinghshe}
			\begin{cases*}
				\partial_tX(t,\alpha)=v(t,X(t,\alpha)),\\
				X(0,\alpha)=\alpha.
			\end{cases*}
		\end{equation}	    	
		Then, it follows from \eqref{zuoyongDelta} that
		\begin{equation}
			\partial_t\left(\Delta_jf\left(t,X(t,\alpha)\right)\right)=\Delta_jg(t,X(t,\alpha))+[v,\Delta_j]\cdot\nabla f(t,X(t,\alpha)).
		\end{equation}
		If we denote the Jacobian determinant of $X(t,\alpha)$ by $J(t,\alpha)=\det(\nabla_\alpha X)(t,\alpha)$, then we have $\partial_tJ(t,\alpha)=J(t,\alpha)(\text{div}v)(t,X(t,\alpha))$. And thus,
		\begin{align*}
			&\mathrel{\phantom{=}}\partial_t\big(J(t,\alpha)\Delta_jf\left(t,X(t,\alpha)\right)\big)\\
			&=J(t,\alpha)\textup{div}v(t,X(t,\alpha))\Delta_jf\left(t,X(t,\alpha)\right)+J(t,\alpha)\Delta_jg(t,X(t,\alpha))\\
			&\mathrel{\phantom{=}}+J(t,\alpha)[v,\Delta_j]\cdot\nabla f(t,X(t,\alpha)),
		\end{align*}
		which together with $J(0,\alpha)=1$ implies that
		\begin{align}
			&\mathrel{\phantom{=}}\big|J(t,\alpha)\Delta_jf\left(t,X(t,\alpha)\right)\big|\nonumber\\
			&\leq |\Delta_jf_0(\alpha)|+\int_{0}^{t}\big|J(\tau,\alpha)\textup{div}v(\tau,X(\tau,\alpha))\Delta_jf\left(\tau,X(\tau,\alpha)\right)\big|d\tau\nonumber\\
			&\mathrel{\phantom{=}}+\int_{0}^{t}\big|J(\tau,\alpha)\Delta_jg(\tau,X(\tau,\alpha))\big|d\tau+\int_{0}^{t}\big|J(\tau,\alpha)[v,\Delta_j]\cdot\nabla f(\tau,X(\tau,\alpha))\big|d\tau\label{jifenhou}.
		\end{align}
		Multiplying $2^{js}$ and taking $l^q$ norm for $j\in\mathbb{Z}$ on both sides of \eqref{jifenhou}, we get by using the Minkowski inequality that
		\begin{align}
			&\mathrel{\phantom{=}}|J(t,\alpha)|\bigg(\sum_{j\in\mathbb{Z} }\left|2^{js}\Delta_jf(t,X(t,\alpha))\right|^{q}\bigg)^{\frac{1}{q}}\nonumber\\
			&\leq\bigg(\sum_{j\in\mathbb{Z} }\left|2^{js}\Delta_jf_0(\alpha)\right|^{q}\bigg)^{\frac{1}{q}}\nonumber+\int_{0}^{t}|J(\tau,\alpha)|(\textup{div}v)(\tau,X(\tau,\alpha))\bigg(\sum_{j\in\mathbb{Z} }\left|2^{js}\Delta_jf(\tau,X(\tau,\alpha))\right|^{q}\bigg)^{\frac{1}{q}}d\tau\nonumber\\
			&\mathrel{\phantom{=}}+\int_{0}^{t}|J(\tau,\alpha)|\bigg(\sum_{j\in\mathbb{Z} }\left|2^{js}\Delta_jg(\tau,X(\tau,\alpha))\right|^{q}\bigg)^{\frac{1}{q}}d\tau\nonumber\\
			&\mathrel{\phantom{=}}+\int_{0}^{t}|J(\tau,\alpha)|\bigg(\sum_{j\in\mathbb{Z} }\left|2^{js}\Big(\big([v,\Delta_j]\cdot\nabla f\big)(\tau,X(\tau,\alpha))\Big)\right|^{q}\bigg)^{\frac{1}{q}}d\tau\label{lqhou},
		\end{align}		
		with the  usual modification if $q=\infty$. Next, taking the $L^p$ norm with respect to $\alpha\in\mathbb{R}^d$ on both sides of \eqref{lqhou}, we get by using the Minkowski inequality that
		\begin{align}
			&\mathrel{\phantom{=}}\Bigg(\int_{\mathbb{R}^d}\Big|J(\tau,\alpha)\Big(\sum_{j\in\mathbb{Z} }\Big|2^{js}\Delta_jf(t,X(t,\alpha))\Big|^{q}\Big)^{\frac{1}{q}}\Big|^pd\alpha\Bigg)^\frac{1}{p}\nonumber\\
			&\leq\|f_0\|_{F^s_{p,q}}+\int_{0}^{t}\Bigg(\int_{\mathbb{R}^d}\Big|J(\tau,\alpha)(\textup{div}v)(\tau,X(\tau,\alpha))\Big(\sum_{j\in\mathbb{Z} }\left|2^{js}\Delta_jf(\tau,X(\tau,\alpha))\right|^{q}\Big)^{\frac{1}{q}}\Big|^pd\alpha\Bigg)^\frac{1}{p}d\tau\nonumber\\
			&\quad+\int_{0}^{t}\Bigg(\int_{\mathbb{R}^d}\Big|J(\tau,\alpha)\Big(\sum_{j\in\mathbb{Z} }\big|2^{js}\Delta_jg(\tau,X(\tau,\alpha))\big|^{q}\Big)^{\frac{1}{q}}\Big|^pd\alpha\Bigg)^\frac{1}{p}d\tau\nonumber\\
			&\quad+\int_{0}^{t}\Bigg(\int_{\mathbb{R}^d}\Big|J(\tau,\alpha)\Big(\sum_{j\in\mathbb{Z} }\big|2^{js}\big(([v,\Delta_j]\cdot\nabla f)(\tau,X(\tau,\alpha))\big)\big|^{q}\Big)^{\frac{1}{q}}\Big|^pd\alpha\Bigg)^\frac{1}{p}d\tau,\label{Lqhou}
		\end{align}	
		with the  usual modification if $p=\infty$. 
		Thus, it follows from \eqref{Lqhou} and the change of variables formula that
		\begin{align}
			\|f\|_{F^s_{p,q}}&\leq\|f_0\|_{F^s_{p,q}}+\int_{0}^{t}\|\textup{div}v\|_{L^\infty}||f||_{F^s_{p,q}}d\tau+\int_{0}^{t}\|g\|_{F^s_{p,q}}d\tau\nonumber\\
			&\quad+\int_{0}^{t}\Big\|\big\|2^{js}[v,\Delta_j]\cdot\nabla f\big\|_{l^q}\Big\|_{L^p}d\tau\label{eq:cedububian}.
		\end{align}
		Then, applying Theorem \ref{prop:commutator-estimates} and Gronwall's inequality completes the proof.
		
		Similarly, the corresponding result in Besov spaces is readily obtained via integration by parts and Theorem \ref{prop:commutator-estimates}. For the sake of simplicity, we omit the details here. Therefore, we complete the proof of Theorem \ref{Thm:a-priori-estimates}.
	\end{proof} 
	
	\section{Local well-posedness for transport equations }
		We now address the local well-posedness result Theorem \ref{thm:solve-transport} for the transport equations \eqref{Eq} with data in  $X^s_{p,q}$. 

	\begin{proof}[Proof of Theorem \ref{thm:solve-transport}]
		For the sake of conciseness, we treat only the case $s<1+\frac{d}{p}$.
		
		We first smooth out the data and the velocity filed $v$ by setting 
		\begin{align*}
			f^n_0\triangleq S_nf_0,\qquad g^n=\rho _n\ast_tS_ng\quad \text{and} \quad v^n=\rho _n\ast_tS_nv,
		\end{align*}
		where $\rho_n\triangleq\rho_n(t)$ stands for a sequence of mollifiers with respect to the time variable. We clearly have $f^n_0\in X^\infty_{p,q},g^n\in C([0,T];X^\infty_{p,q}),v^n\in C([0,T]\times\mathbb{R}^d)$ and $\nabla v^n\in C([0,T];X^\infty_{p,q})$ with $X^\infty_{p,q}\triangleq\cap_{s\in\mathbb{R}}X^s_{p,q}$. Moreover,  $f^n_0$ is uniformly bounded in $X^s_{p,q}$, $g^n$ is uniformly bounded in $L^1(0,T;X^\infty_{p,q})$, $v^n$ is uniformly bounded in  $L^\rho(0,T;X^{-M}_{\infty,\infty})$ and $\nabla v^n$ is uniformly  bounded in  $L^1(0,T;X^{\frac{d}{p}}_{p,\infty}\cap L^\infty)$.
		
		Let $f^n$ be the solution to the following equation:
		\begin{equation}\label{nEq}
			\begin{cases*}
				\partial_tf^n+v^n\cdot\nabla f^n=g^n,\\
				f^n|_{t=0}=f^n_0.
			\end{cases*}		
		\end{equation}
		Clearly, $f^n$ is smooth, and according to Theorem \ref{Thm:a-priori-estimates}, one has 
		\begin{equation}
			\|f^n(t)\|_{X^s_{p,q}}\leq e^{C\int_{0}^{t}Z^n(\tau)d\tau}\bigg(\|f^n_0\|_{X^s_{p,q}}+\int_{0}^{t}\|g^n(\tau)\|_{X^s_{p,q}}e^{-C\int^\tau_0Z^n(\tau')d\tau'}d\tau\bigg)
		\end{equation}
		with $Z^n(t)\triangleq\|\nabla v^n(t)\|_{X^{\frac{d}{p}}_{p,\infty}\cap L^\infty}$.
		
		Thus, in view of the uniform bounds for $f^n_0,g^n$ and $v^n$, we conclude that the sequence $\{f^n\}_{n\in\mathbb{N}}$ is uniformly bounded in $C([0,T];X^s_{p,q})$. In order to prove the convergence of a subsequence, we appeal to compactness arguments. First, notice that
		\begin{equation}\label{equation12}
			\partial_tf^n-g^n=-v^n\cdot \nabla f^n.
		\end{equation}
		Since $\nabla f^n$ is uniformly bounded in $L^\infty(0,T;X^{s-1}_{p,q})$, and for small enough $\varepsilon$, $v^n$ is uniformly bounded in $L^\alpha(0,T;X^{\frac{d}{p}+1-\varepsilon}_{p+\varepsilon,\infty})$ for some $\alpha>1$  (interpolate between the uniform bounds in $L^1(0,T;X^{\frac{d}{p}+1}_{p,\infty})$ and in $L^\rho(0,T;X^{-M}_{\infty,\infty} )$ for $v^n$), one can conclude by appealing  Lemmas \ref{lem:Triebel-Lizorkin-properties} $(i)$ and \ref{lem:product} that the right hand-side of \eqref{equation12} is uniformly bounded in $L^\alpha(0,T;X^{-m}_{p,\infty})$ for some large enough $m>0$.  Integrating in time and denoting $\bar{f}^n(t)\triangleq f^n(t)-\int_{0}^tg^n(\tau)d\tau$,  we thus gather that there exists some $\beta>0$ such that the sequence  $\{\bar{f}^n\}_{n\in\mathbb{N}}$ is uniformly bounded in $C^\beta([0,T];X^{-m}_{p,\infty})$, hence uniformly equicontinuous with values in $X^{-m}_{p,\infty}$.
		
		Next, assuming that $m$ is large enough, observe that the map $f\mapsto\phi f$ is compact from $X^s_{p,q}\to X^{-m}_{p,\infty}$ for all $\phi\in C^\infty_c$ (by Lemma \ref{lem:Triebel-Lizorkin-properties} (ii)). Combining the Arzel\`{a}-Ascoli theorem and the Cantor diagonal process thus ensures that, up to a subsequence, the sequence $\{\bar{f}^n\}_{n\in\mathbb{N}}$ converges in $S'$ to some distribution $\bar{f}$ such that $\phi\bar{f}\in C([0,T];X^{-m}_{p,\infty})$ for all $\phi\in C^\infty_c$. 
		
		Finally, appealing once again to the uniform bounds in $L^\infty(0,T;X^s_{p,q})$ and the Fatou property (Lemma \ref{lem:Triebel-Lizorkin-properties} ($iv$)) for Triebel-Lizorkin or Besov spaces, we get $\bar{f}\in L^\infty(0,T;X^{s}_{p,q})$. By an interpolation argument, together with the bounds in $L^\infty(0,T;X^{s}_{p,q})$ for $\{\bar{f}^n\}_{n\in\mathbb{N}}$, we find that $\phi\bar{f}^n\mapsto \phi\bar{f}$ in $C([0,T];X^{s-\varepsilon}_{p,\infty})$ for all $\varepsilon>0$ and $\phi\in C^\infty_c$ so that we may pass to the limit in the equation for $f^n$, in the sense of distributions. Besides, the sequences $\{f_0^n\}_{n\in\mathbb{N}}$, $\{g^n\}_{n\in\mathbb{N}}$ and  $\{v^n\}_{n\in\mathbb{N}}$ converge respectively to $f_0$, $g$, and $v$, which may be easily deduced from their definitions. We conclude that the function $f\triangleq\bar{f}+\int_{0}^tg(\tau)d\tau$ is a solution to \eqref{Eq}.
		
		We still have to prove that $f\in C([0,T];X^s_{p,q})$ in the case where $q< \infty$. We only prove the case of Triebel-Lizorkin spaces here, the Besov case is analogous and follows readily. Just by looking at the equation \eqref{Eq}, it is easy to get $\partial_tf\in L^1(0,T;F^{-M'}_{p,\infty})$ for some large enough $M'$. Hence $f\in C([0,T];F^{-M'}_{p,\infty})$, whence $S_nf\in C([0,T];F^s_{p,q})$ for all $n\in\mathbb{N}$. Note that 
		\begin{equation*}
			\Delta_j(f-S_nf)=\sum_{\substack{|j-j'|\leq 1\\j'\geq n}}\Delta_j\Delta_{j'}f.
		\end{equation*}
		For $1<p,q<\infty$, using Lemma \ref{lem:miaoguji}, Young's inequality and Lemma \ref{lem:Fguji}, we have 
		\begin{align}
			\|f-S_nf\|_{F^s_{p,q}}&=\Big\|\big\|\sum_{\substack{|j-j'|\leq 1\\j'\geq n}}2^{js}\Delta_j\Delta_{j'}f\big\|_{l^q}\Big\|_{L^p}\nonumber\\
			&\leq C\Big\|\big\|\sum_{\substack{|j-j'|\leq 1\\j'\geq n}}2^{(j-j')s}2^{j's}M(\Delta_{j'}f)\big\|_{l^q}\Big\|_{L^p}\nonumber\\
			&\leq  C\Big\|\big\|2^{j's}\Delta_{j'}f\big\|_{l^q(j'\geq n)}\Big\|_{L^p}.\label{equation19}		
		\end{align}
		The above \eqref{equation19} also holds true for $p = 1$ or $q=1$, since Lemma \ref{lem:guoguji}, Young's inequality and Lemma \ref{lem:Fguji}. Similar to the proof of \eqref{eq:cedububian}, for any $n\in\mathbb{N}$, consider  $j'\geq n$ instead of $j'\in\mathbb{Z}$, we can show that
		\begin{align*}
			\Big\|\big\|2^{j's}\Delta_{j'}f\big\|_{l^q(j'\geq n)}\Big\|_{L^p}&\leq \Big\|\big\|2^{j's}\Delta_{j'}f_0\big\|_{l^q(j'\geq n)}\Big\|_{L^p}+\int_{0}^t\|\text{div} v\|_{L^\infty}\Big\|\big\|2^{j's}\Delta_{j'}f\big\|_{l^q(j'\geq n)}\Big\|_{L^p}d\tau\nonumber\\
			&\quad+\int_{0}^t\Big\|\big\|2^{j's}\Delta_{j'}g\big\|_{l^q(j'\geq n)}\Big\|_{L^p}+\int_{0}^{t}\Big\|\big\|2^{j's}[v,\Delta_j]\cdot\nabla f\big\|_{l^q(j'\geq n)}\Big\|_{L^p}d\tau,
		\end{align*}
		which along with \eqref{equation19} and Gronwall's inequality yields
		\begin{align}
			\|f-S_nf\|_{L^\infty_T(F^s_{p,q})}&\leq Ce^{C\int_{0}^{T}\|\nabla v\|_{L^\infty}d\tau}\Big(\Big\|\big\|2^{j's}\Delta_{j'}f_0\big\|_{l^q(j'\geq n)}\Big\|_{L^p}+\int^T_0\Big\|\big\|2^{j's}\Delta_{j'}g\big\|_{l^q(j'\geq n)}\Big\|_{L^p}d\tau\nonumber\\
			&\quad+\int_{0}^T\Big\|\big\|2^{j's}[v,\Delta_j]\cdot\nabla f\big\|_{l^q(j'\geq n)}\Big\|_{L^p}d\tau\Big)\label{equation3}.
		\end{align}
		The first term of right hand-side in \eqref{equation3} clearly tends to zero when $n$ goes to infinity. Thanks to Remark \ref{remark:jiaohaunziquyu0}, the commutator term in the third term of right hand-side of \eqref{equation3} tends to zero when $n$ goes to infinity. So the terms in the integrals approach to zero for almost every $t$. Hence, by virtue of Lebesgue's dominated convergence theorem, $\|f-S_nf\|_{L^\infty_T(F^s_{p,q})}$ tends to zero when $n$ goes to infinity. This achieves to proving that $f\in C([0,T];F^s_{p,q})$ in the case $q<\infty$.
		
		When $q=\infty$, we use that for any $s'<s$, we have the embedding $F^s_{p,\infty}\hookrightarrow F^{s'}_{p,1}$ so that the above argument may be repeated in the space $F^{s'}_{p,1}$, this yields  $f\in C([0,T];F^{s'}_{p,1})$.
		
		For the  uniqueness and continuity with respect to the initial data, if we are given  $(f^1,f^2)\in L^\infty(0,T;X^s_{p,q}\times X^s_{p,q})\cap C([0,T];\mathscr{S}'\times\mathscr{S}')$ two solutions to \eqref{Eq} with initial data $f^1_0,f^2_0\in X^s_{p,q}$. Denote $w=f^1-f^2$, then $w\in C([0,T];X^{s-1}_{p,q})$ solves the following  transport equation:
		\begin{equation}
			\begin{cases}
				\partial_tw+v\cdot\nabla w=0,\\
				w|_{t=0}=f^1_0-f^2_0.
			\end{cases}			
		\end{equation} 
		In view of Theorem \ref{Thm:a-priori-estimates}, we have for every $t\in[0,T],$
		\begin{equation}
			||f^1-f^2||_{X^{s}_{p,q}}\leq e^{C\int_{0}^tZ(\tau)d\tau}\|f^1_0-f^2_0\|_{X^{s}_{p,q}},
		\end{equation}
		which implies the uniqueness of the solution and its continuous dependence.
		Therefore, we complete the proof of Theorem \ref{thm:solve-transport}. 
	\end{proof}

	\section{Solving the two-component Euler-poincar\'{e} system}
	In this section, we will establish the local well-posedness for system \eqref{eq} or \eqref{transport-system} in the general Triebel-Lizorkin spaces by using the Friedrichs regularization method and transport equations theory developed in this paper.

	\begin{proof}[Proof of Theorem \ref{thm:well-poseness}]
		\textbf{Existence:} We prove it in the following four steps.\\
		\textbf{Step 1. Approximate solutions.}
		
		Staring form $(u^0,\gamma^0)\triangleq (0,0)$, we then define by induction a sequence of smooth functions $(u^n,\gamma^n)_{n\in\mathbb{N}}$ by solving the following transport equations:
		\begin{equation}\label{approximate-solution-system}
			\begin{cases}
				\partial_tu^{n+1}+u^n\cdot\nabla  u^{n+1}=F_1(u^n,\gamma^n)\triangleq F^n_1(t,x),\\
				\partial_t\gamma^{n+1}+u^n\cdot\nabla  \gamma^{n+1}=F_2(u^n,\gamma^n)\triangleq F^n_2(t,x),\\
				u^{n+1}(0,x)=u^{n+1}_0(x)\triangleq S_{n+1}u_0(x),\\
				\gamma^{n+1}(0,x)=\gamma^{n+1}_0(x)\triangleq S_{n+1}\gamma_0(x).
			\end{cases}
		\end{equation}
		where $F_1(u^n,\gamma^n)$ and $F_2(u^n,\gamma^n)$ are defined by \eqref{eq:F1} and \eqref{eq:F2}. Since all the data belong to $F^\infty_{p,q}$, Theorem \ref{thm:solve-transport} enables us to show by induction
		that for all $n\in\mathbb{N}$, the above system has a global solution which belongs to $C(\mathbb{R}^+;F^\infty_{p,q}\times F^\infty_{p,q})$.\\
		\textbf{Step 2. Uniform bounds.}
		
		According to Theorem \ref{Thm:a-priori-estimates}, we have the following inequalities for all $n\in\mathbb{N}$:
		\begin{equation}\label{eq:equation8}
			\|u^{n+1}(t)\|_{F^{s}_{p,q}}\leq Ce^{C\int_{0}^t\|u^n(\tau)\|_{F^s_{p,q}}d\tau}\Big(\|u^{n+1}_0\|_{F^{s}_{p,q}}+\int_{0}^te^{-C\int_{0}^\tau\|u^n(\tau')\|_{F^s_{p,q}}d\tau'}\|F^n_1(\tau)\|_{F^{s}_{p,q}}d\tau\Big),
		\end{equation}
		\begin{equation}\label{eq:equation9}
			\|\gamma^{n+1}(t)\|_{F^{s-1}_{p,q}}\leq Ce^{C\int_{0}^t\|u^n(\tau)\|_{F^{s}_{p,q}}d\tau}\Big(\|\gamma^{n+1}_0\|_{F^{s-1}_{p,q}}+\int_{0}^te^{-C\int_{0}^\tau\|u^n(\tau')\|_{F^{s}_{p,q}}d\tau'}\|F^n_2(\tau)\|_{F^{s-1}_{p,q}}d\tau\Big).
		\end{equation}
		Next, we estimate the $\|F_1^n(t)\|_{F^s_{p,q}}$ and $\|F_2^n(t)\|_{F^{s-1}_{p,q}} $. By definition, we have
		\begin{align*}
		    F_1^n(t,x)&=-(I-\Delta)^{-1}\text{div}\Big(\nabla u^n\nabla u^n+\nabla u^n(\nabla u^n)^T-(\nabla u^n)^T\nabla u^n-\nabla u^n(\text{div}u^n)\\
		    &\quad+\frac{1}{2}|\nabla u^n|^2I+\frac{1}{2}(\gamma^n)^2I+\gamma^nI\Big)-(I-\Delta)^{-1}\Big(u^n(\text{div}u^n)+(\nabla u^n)^T\cdot u^n\Big)\\
			&\triangleq R_1+R_2,
		\end{align*}
		and
		\begin{equation*}
			F_2^n(t,x)=-\gamma^n\textup{div }u^n-\text{div}u^n.
		\end{equation*}
		Using Lemma \ref{lem:Triebel-Lizorkin-properties} $(vi)$ and the fact that when $s>1+\frac{d}{p}$ or $s=1+d$ with $p=1$, $F^{s-1}_{p,q}\hookrightarrow L^\infty$, the $F^{s-1}_{p,q}$ is an algebra, we obtain
		\begin{align}
			\|R_1\|_{F^{s}_{p,q}}
			&\leq C\big(\|\nabla u^n\nabla u^n\|_{F^{s-1}_{p,q}}+\|\nabla u^n(\nabla u^n)^T\|_{F^{s-1}_{p,q}}+\|(\nabla u^n)^T\nabla u^n\|_{F^{s-1}_{p,q}}\nonumber\\
			&\quad+\|\nabla u^n(\text{div}u^n)\|_{F^{s-1}_{p,q}}+\frac{1}{2}\|\nabla u^n\|^2_{F^{s-1}_{p,q}}+\frac{1}{2}\| \gamma^n\|^2_{F^{s-1}_{p,q}}+\| \gamma^n\|_{F^{s-1}_{p,q}}\big)\nonumber\\
			&\leq C\big(\|u^n\|^2_{F^{s}_{p,q}}+\|\gamma^n\|^2_{F^{s-1}_{p,q}}+\|\gamma^n\|_{F^{s-1}_{p,q}}\big).\label{eq:equation10}
		\end{align}
		From Lemma \ref{lem:product}, we deduce that for $s>\max(1+\frac{d}{p},\frac{3}{2})$,
		\begin{align}
			\|R_2\|_{F^{s}_{p,q}}&\leq C\big(\|u^n(\text{div}u^n)\|_{F^{s-2}_{p,q}}+\|(\nabla u^n)^T\cdot u^n\|_{F^{s-2}_{p,q}}\big)\nonumber\\
			&\leq C\big(\|u^n\|_{F^{s-1}_{p,q}}\|\text{div}u^n\|_{F^{s-2}_{p,q}}+\|\nabla u^n\|_{F^{s-2}_{p,q}}\|u^n\|_{F^{s-1}_{p,q}}\big)\nonumber\\
			&\leq C\|u^n\|^2_{F^{s}_{p,q}}.\label{eq:equation52}
 		\end{align}
 		For the critical case $s=1+d$ with $p=1$, an application of Remark \ref{remark:product-critical} yields, for $d\geq 2$,
 		\begin{align}
 			\|R_2\|_{F^{1+d}_{1,q}}&\leq C\big(\|u^n(\text{div}u^n)\|_{F^{d-1}_{1,q}}+\|(\nabla u^n)^T\cdot u^n\|_{F^{d-1}_{1,q}}\big)\nonumber\\
 			&\leq C\big(\|u^n\|_{F^{d}_{1,q}}\|\text{div}u^n\|_{F^{d-1}_{1,q}}+\|\nabla u^n\|_{F^{d-1}_{1,q}}\|u^n\|_{F^{d}_{1,q}}\big)\nonumber\\
 			&\leq C\|u^n\|^2_{F^{1+d}_{1,q}}.\label{eq:equation55}
 		\end{align}
 		Regarding $\|F^n_2\|_{F^{s-1}_{p,q}}$, it is straightforward to obtain
		\begin{align}
			\|F^n_2\|_{F^{s-1}_{p,q}}&\leq C\big(\|\gamma^n\|_{F^{s-1}_{p,q}}\|\text{div}u^n\|_{F^{s-1}_{p,q}}+\|\text{div}u^n\|_{F^{s-1}_{p,q}}\big)\nonumber\\
			&\leq C\big(\|\gamma^n\|_{F^{s-1}_{p,q}}\|u^n\|_{F^{s}_{p,q}}+\|u^n\|_{F^{s}_{p,q}}\big),\label{eq:equation53}
		\end{align}
		under the condition that either $s>1+\frac{d}{p}$, or $s=1+d$ with $p=1$.
		By combining \eqref{eq:equation10}-\eqref{eq:equation53}, for $s>\max(1+\frac{d}{p},\frac{3}{2})$ or $s=1+d$ with $p=1$ and $d\geq 2$, we have
		\begin{equation}\label{eq:equation62}
			\|F^n_1\|_{F^s_{p,q}}+\|F^n_2\|_{F^{s-1}_{p,q}}\leq C(\|u^n\|_{F^s_{p,q}}+\|\gamma^n\|_{F^{s-1}_{p,q}}+1)(\|u^n\|_{F^s_{p,q}}+\|\gamma^n\|_{F^{s-1}_{p,q}}).
		\end{equation}
		Set
		\[ \Gamma^n\triangleq \|u^{n}\|_{F^s_{p,q}}+\|\gamma^{n}\|_{F^{s-1}_{p,q}}.\]
		Then, in view of \eqref{eq:equation8}, \eqref{eq:equation9}, and \eqref{eq:equation62}, we get
		\begin{equation}\label{eq:equation54}
			\Gamma^{n+1}(t)\leq Ce^{CU^n(t)} \Big(\|u_0\|_{F^s_{p,q}}+\|\gamma_0\|_{F^{s-1}_{p,q}}+\int_{0}^te^{-CU^n(\tau)}(\Gamma^n(\tau)+1)(\Gamma^n(\tau))d\tau\Big),
		\end{equation}
		where $U^n(t)\triangleq \int_{0}^t\| u^n(\tau)\|_{F^s_{p,q}}d\tau$. Let us fix $T>0$ such that $2C^2(\|u_0\|_{F^s_{p,q}}+\|\gamma_0\|_{F^{s-1}_{p,q}})T<1$, and suppose that
		\begin{equation}\label{eq:Gamma-n-estimate}
			\Gamma^n\leq \frac{C(\|u_0\|_{F^s_{p,q}}+\|\gamma_0\|_{F^{s-1}_{p,q}})}{1-2C^2(\|u_0\|_{F^s_{p,q}}+\|\gamma_0\|_{F^{s-1}_{p,q}})t},\qquad\forall t\in[0,T],\ \forall n\in\mathbb{N}.
		\end{equation} 
		Plugging \eqref{eq:Gamma-n-estimate} in \eqref{eq:equation54} eventually yields
		\begin{align}
			\Gamma^{n+1}&\leq\frac{C(\|u_0\|_{F^s_{p,q}}+\|\gamma_0\|_{F^{s-1}_{p,q}})}{\sqrt{1-2C^2(\|u_0\|_{F^s_{p,q}}+\|\gamma_0\|_{F^{s-1}_{p,q}})t}}+\frac{C}{\sqrt{1-2C^2(\|u_0\|_{F^s_{p,q}}+\|\gamma_0\|_{F^{s-1}_{p,q}})t}}\nonumber\\
			&\quad\times\int_{0}^{t}\frac{C(\|u_0\|_{F^s_{p,q}}+\|\gamma_0\|_{F^{s-1}_{p,q}})^2}{\big(1-2C^2(\|u_0\|_{F^s_{p,q}}+\|\gamma_0\|_{F^{s-1}_{p,q}})\tau\big)^\frac{3}{2}}d\tau+\frac{C}{\sqrt{1-2C^2(\|u_0\|_{F^s_{p,q}}+\|\gamma_0\|_{F^{s-1}_{p,q}})t}}\nonumber\\
			&\quad\times\int_{0}^{t}\frac{C(\|u_0\|_{F^s_{p,q}}+\|\gamma_0\|_{F^{s-1}_{p,q}})}{\big(1-2C^2(\|u_0\|_{F^s_{p,q}}+\|\gamma_0\|_{F^{s-1}_{p,q}})\tau\big)^\frac{1}{2}}d\tau\nonumber\\
			&\leq \frac{C(\|u_0\|_{F^s_{p,q}}+\|\gamma_0\|_{F^{s-1}_{p,q}})}{1-2C^2(\|u_0\|_{F^s_{p,q}}+\|\gamma_0\|_{F^{s-1}_{p,q}})t}.
		\end{align}		
		Therefore, we obtain that the sequence $(u^n,\gamma^n)_{n\in\mathbb{N}}$ is uniformly bounded in the space $C([0,T];F^s_{p,q}\times F^{s-1}_{p,q})$. This clearly entails that $(u^n\cdot\nabla u^{n},u^n\cdot\nabla \gamma^{n})$ is uniformly bounded in $C([0,T];F^{s-1}_{p,q}\times F^{s-2}_{p,q})$, and right-hand side of the first and second equations of \eqref{approximate-solution-system} have been shown to be uniformly bounded in $C([0,T];F^s_{p,q}\times F^{s-1}_{p,q})$. We can conclude that $(u^n,\gamma^n)_{n\in\mathbb{N}}$ is uniformly bounded in $E^s_{p,q}(T)\times E^{s-1}_{p,q}(T)$.\\	
		\textbf{Step 3. Convergence.}
		
		We are going to show that the sequence $(u^n,\gamma^n)_{n\in\mathbb{N}}$ is a Cauchy sequence in $C([0,T];F^{s-1}_{p,q}\times F^{s-2}_{p,q})$. For that purpose, we note that for all $(m,n)\in\mathbb{N}^2 $, we have
		\begin{equation}
			(\partial_t+u^{n+m}\cdot\nabla)(u^{n+m+1}-u^{n+1}) 
			= F_1^{n+m}(t,x)-F^n_1(t,x)+(u^{n}-u^{n+m})\cdot\nabla u^{n+1},
		\end{equation}
		and
		\begin{equation}
			(\partial_t+u^{n+m}\cdot\nabla)(\gamma^{n+m+1}-\gamma^{n+1}) 
			= F_2^{n+m}(t,x)-F^n_2(t,x)+(u^{n}-u^{n+m})\cdot\nabla \gamma^{n+1},
		\end{equation}
		where $F^{n+m}_i(t,x)\triangleq F_i(u^{n+m},\gamma^{n+m}), i=1,2$. Similar to the proof of \eqref{eq:equation54}, for $s>\max(1+\frac{d}{p},\frac{3}{2})$ or $s=1+d$ with $p=1$ and $d\geq 2$, one infers
		\begin{equation*}
			\begin{aligned}
				\varLambda^{n+m+1}_{n+1}(t)
				&\leq Ce^{CU^{n+m}(t)}\Big(\|(u_0^{n+m+1}-u_0^{n+1})(t)\|_{F^{s-1}_{p,q}}+\|(\gamma_0^{n+m+1}-\gamma_0^{n+1})(t)\|_{F^{s-2}_{p,q}}\\
				&\quad+\int_{0}^te^{-CU^{n+m}(\tau)}\varLambda^{n+m}_n(\tau)\big(\Gamma^n(\tau)+\Gamma^{n+1}(\tau)+\Gamma^{n+m}(\tau)+1\big)d\tau\Big),
			\end{aligned}
		\end{equation*}
		where $\varLambda^{n+m}_{n}(t)\triangleq\|(u^{n+m}-u^{n})(t)\|_{F^{s-1}_{p,q}}+\|(\gamma^{n+m}-\gamma^{n})(t)\|_{F^{s-2}_{p,q}}$.
		Recalling the definition of $(u^n_0,\gamma^n_0)$, we obtain
		\[ u^{n+m+1}_0-u^{n+1}=\sum_{j=n+1}^{n+m}\Delta_ju_0,\qquad \gamma^{n+m+1}_0-\gamma^{n+1}=\sum_{j=n+1}^{n+m}\Delta_j\gamma_0.\]
		Thanks to \eqref{a.o.c}, Young's inequality and Lemmas  \ref{lem:Fguji}-\ref{lem:miaoguji}, we have for $1\leq p<\infty$ and $1\leq q\leq \infty$,
		\begin{align*}
			\|\Delta_{n}u_0\|_{F^{s-1}_{p,q}}&=\Big\|\big\|2^{j(s-1)}\Delta_j\Delta_{n}u_0\big\|_{l^q}\Big\|_{L^p}=\Big\|\big(\sum_{|j-n|\leq 1}(2^{jq(s-1)}|\Delta_j\Delta_{n}u_0|)^q\big)^{\frac{1}{q}}\Big\|_{L^p}\\
			&\lesssim \Big\|\big(\sum_{|j-n|\leq 1}(2^{sq(j-n)}2^{-jq}M(2^{nsq}\Delta_{n}u_0))^q\big)^{\frac{1}{q}}\Big\|_{L^p}\\
			&\lesssim 2^{-(n-1)}\Big\|\big\|2^{nsq}\Delta_nu_0\big\|_{l^q}\Big\|_{L^p}=2^{-(n-1)}\|u_0\|_{F^s_{p,q}}.
		\end{align*}
		The end point cases where $p=1$ or $q=1$ can be handled by replacing
		Lemma \ref{lem:miaoguji} with Lemma \ref{lem:guoguji} in the above argument. Meanwhile, one readily finds that $\|\Delta_{n}u_0\|_{F^{s-1}_{\infty,\infty}}\lesssim2^{-(n-1)}\|u_0\|_{F^s_{\infty,\infty}}$. Thus, we get
		\begin{equation*}
			\|u^{n+m+1}_0-u^{n+1}_0\|_{F^{s-1}_{p,q}}\lesssim2^{-n}\left\|u_0\right\|_{F^{s}_{p,q}},\quad\|\gamma^{n+m+1}_0-\gamma^{n+1}_0\|_{F^{s-2}_{p,q}}\lesssim2^{-n}\left\|\gamma_0\right\|_{F^{s-1}_{p,q}}.
		\end{equation*} 
		Since $(u^n,\gamma^n)_{n\in\mathbb{N}}$ is uniformly bounded in $E^s_{p,q}(T)\times E^{s-1}_{p,q}(T)$, we get a constant $C_T$ independent of $n,m$ and such that for all $t\in[0,T]$, we have
		\begin{equation*}
			\varLambda^{n+m+1}_{n+1}(t)\leq C_T\big(2^{-n}+\int_{0}^{t}\varLambda^{n+m}_{n}(\tau)d\tau\big).
		\end{equation*}
		Arguing by induction, one can find 
		\begin{equation*}
			\varLambda^{n+m+1}_{n+1}(t)\leq\frac{(TC_T)^{n+1}}{(n+1)!}+C_T\sum_{k=0}^{n}2^{-(n-k)}\frac{(TC_k)^k}{k!}.
		\end{equation*}
		This implies that $(u^n,\gamma^n)_{n\in\mathbb{N}}$ is a Cauchy sequence in $C([0,T];F^{s-1}_{p,q}\times F^{s-2}_{p,q})$, whence it converges to some limit function $(u,\gamma)\in C([0.T];F^{s-1}_{p,q}\times F^{s-2}_{p,q})$.
		
		\textbf{Step 4. Conclusion.}
		
		Finally, we prove that $(u,\gamma)\in E^s_{p,q}(T)\times E^{s-1}_{p,q}(T)$ and satisfies system \eqref{transport-system}. Since $(u^n,\gamma^n)_{n\in\mathbb{N}}$ is uniformly bounded in $L^\infty(0,T;F^s_{p,q}\times F^{s-1}_{p,q})$, then the Fatou property (Lemma \ref{lem:Triebel-Lizorkin-properties} ($iv$)) guarantees that $(u,\gamma)$ also belongs to $L^\infty(0,T;F^s_{p,q}\times F^{s-1}_{p,q})$. As $(u^n,\gamma^n)_{n\in\mathbb{N}}\to(u,\gamma)$ in $C([0.T];F^{s-1}_{p,q}\times F^{s-2}_{p,q})$, an interpolation argument insures that convergence actually holds in $C([0.T];F^{s'}_{p,q}\times F^{s'-1}_{p,q})$ for any $s'<s$. It is easy to pass to the limit in system \eqref{approximate-solution-system} and to conclude that $u$ is indeed a solution to system \eqref{transport-system}.
		
		Since $(u,\gamma)\in L^\infty(0,T;F^s_{p,q}\times F^{s-1}_{p,q})$, it follows that $F_1(t,x)$ and $F_2(t,x)$ also belong to $ L^\infty(0,T;F^s_{p,q})$ and $ L^\infty(0,T;F^{s-1}_{p,q})$ , respectively. In the case $q<\infty$, Theorem \ref{thm:solve-transport}  enables us to conclude that $(u,\gamma)\in C([0,T];F^s_{p,q}\times F^{s-1}_{p,q})$. Moreover ,we can find $\partial_tu $ and $\partial_t\gamma $ is in $C([0,T];F^{s-1}_{p,q}\times F^{s-2}_{p,q})$ if $q<\infty$, and $L^\infty(0,T;F^{s-1}_{p,q}\times F^{s-2}_{p,q})$ otherwise. Hence, $(u,\gamma)\in E^s_{p,q}(T)\times E^{s-1}_{p,q}(T)$.\\
		
		\textbf{Uniqueness:} Let us consider $(v,\eta)\in L^\infty(0,T;F^s_{p,q}\times F^{s-1}_{p,q})\cap C([0,T];\mathscr{S}'\times \mathscr{S}')$ is another solution to system \eqref{transport-system} with the initial data $(v_0,\eta_0)$. Denote $(w,\theta)\triangleq(v-u,\eta-\gamma)$ and $(w_0,\theta_0)\triangleq(v_0-u_0,\eta_0-\gamma_0) $. Then for every $t\in[0,T]$, we have
		\begin{equation}\label{eq:difference-equation}
			\begin{cases}
				\partial_tw+u\cdot\Delta w=P(t,x),\\
				\partial_t\theta+u\cdot\Delta\theta=Q(t,x),\\
				w(0,x)=w_0(x),\\
				\theta(0,x)=\theta_0(x),
			\end{cases}
		\end{equation}
		where
		\begin{align*}
			 P(t,x)
			&=-w\cdot\nabla v-(I-\Delta)^{-1}\big(w(\text{div}v)+u(\text{div}w)+w\cdot\nabla v^T+u\cdot\nabla w^T\big)\\
			&\quad-(I-\Delta)^{-1}\text{div}\big(\nabla w(\nabla v+\nabla v^T)+(\nabla u-\nabla u^T)\nabla w+\nabla u\nabla w^T-\nabla w^T\nabla v\big)\\
			&\quad-(I-\Delta)^{-1}\text{div}\big(-\nabla w(\text{div}v)-\nabla u(\text{div}w)\big)\\
			&\quad-(I-\Delta)^{-1}\text{div}\bigg(\frac{1}{2}\big(\nabla(u+v):\nabla w+(\eta+\gamma)\theta\big)I+\theta I\bigg)\\
			&\triangleq P_1+P_2+P_3+P_4,
		\end{align*}
		and
		\begin{equation}
			 Q(t,x)=-w\cdot\nabla\eta-\theta(\text{div}v)-\gamma(\text{div}w)-\text{div}w.
		\end{equation}
		Analogous to the proofs of \eqref{eq:equation10}-\eqref{eq:equation55}, we obtain that for $s>\max(1+\frac{d}{p},\frac{3}{2})$ (or $s=1+d$ with $p=1$ and $d\geq 2$),
		\begin{equation}
			\|w\cdot\nabla v\|_{F^{s-1}_{p,q}}\leq C\|w\|_{F^{s-1}_{p,q}}\|\nabla v\|_{F^{s-1}_{p,q}}\leq C\|w\|_{F^{s-1}_{p,q}}\| v\|_{F^{s}_{p,q}},
		\end{equation}
		\begin{equation}
			\|w\cdot\nabla \eta\|_{F^{s-2}_{p,q}}\leq C\|w\|_{F^{s-1}_{p,q}}\|\nabla \eta\|_{F^{s-2}_{p,q}}\leq C\|w\|_{F^{s-1}_{p,q}}\| \eta\|_{F^{s-1}_{p,q}},
		\end{equation}
		\begin{equation}
			\|\theta(\text{div}v)\|_{F^{s-2}_{p,q}}\leq C \|\theta\|_{F^{s-2}_{p,q}}\|\text{div}v\|_{F^{s-1}_{p,q}}\leq C \|\theta\|_{F^{s-2}_{p,q}}\|v\|_{F^{s}_{p,q}},
		\end{equation}
		\begin{equation}
			\|\gamma(\text{div}w)\|_{F^{s-2}_{p,q}}\leq C \|\text{div}w\|_{F^{s-2}_{p,q}}\|\gamma\|_{F^{s-1}_{p,q}}\leq C \|w\|_{F^{s-1}_{p,q}}\|\gamma\|_{F^{s-1}_{p,q}},
		\end{equation}
		\begin{equation}
			\|\text{div}w\|_{F^{s-2}_{p,q}}\leq C \|w\|_{F^{s-1}_{p,q}}.
		\end{equation}
		Likewise, for $s>\max(1+\frac{d}{p},\frac{3}{2})$ (or $s=1+d$ with $p=1$ and $d\geq 2$), Lemma \ref{lem:Triebel-Lizorkin-properties} $(vi)$ implies the following estimates:
		\begin{align}
			\|P_1\|_{F^{s-1}_{p,q}}&=\|(I-\Delta)^{-1}\big(w(\text{div}v)+u(\text{div}w)+w\cdot\nabla v^T+u\cdot\nabla w^T\big)\|_{F^{s-1}_{p,q}}\nonumber\\
			&\leq C\|w(\text{div}v)+u(\text{div}w)+w\cdot\nabla v^T+u\cdot\nabla w^T\|_{F^{s-3}_{p,q}}\nonumber\\
			&\leq C\|w(\text{div}v)+u(\text{div}w)+w\cdot\nabla v^T+u\cdot\nabla w^T\|_{F^{s-2}_{p,q}}\nonumber\\
			&\leq C\|w\|_{F^{s-1}_{p,q}} (\|v\|_{F^{s}_{p,q}}+\|u\|_{F^{s}_{p,q}}),
		\end{align}
	    \begin{align}
			\|P_2\|_{F^{s-1}_{p,q}}&\leq C\|\nabla w(\nabla v+\nabla v^T)+(\nabla u-\nabla u^T)\nabla w+\nabla u\nabla w^T-\nabla w^T\nabla v\|_{F^{s-2}_{p,q}}\nonumber\\
			&\leq C\|\nabla w\|_{F^{s-2}_{p,q}}\big(\|\nabla v\|_{F^{s-1}_{p,q}}+\|\nabla u\|_{F^{s-1}_{p,q}}\big)\nonumber\\
			&\leq C\| w\|_{F^{s-1}_{p,q}}\big(\| v\|_{F^{s}_{p,q}}+\| u\|_{F^{s}_{p,q}}\big),
		\end{align}
		\begin{equation}
			\|P_3\|_{F^{s-1}_{p,q}}
			\leq C\|-\nabla w(\text{div}v)-\nabla u(\text{div}w)\|_{F^{s-2}_{p,q}}\leq C\|w\|_{F^{s-1}_{p,q}}(\|v\|_{F^{s}_{p,q}}+\|u\|_{F^{s}_{p,q}}),
		\end{equation}
		\begin{align}
			\|P_4\|_{F^{s-1}_{p,q}}&\leq C\|\frac{1}{2}\big(\nabla(u+v):\nabla w+(\eta+\gamma)\theta\big)I+\theta I\|_{F^{s-2}_{p,q}}\nonumber\\
			&\leq C\big( \|w\|_{F^{s-1}_{p,q}}(\|v\|_{F^{s}_{p,q}}+\|u\|_{F^{s}_{p,q}})+\|\theta\|_{F^{s-2}_{p,q}}(\|\gamma\|_{F^{s-1}_{p,q}}+\|\eta\|_{F^{s-1}_{p,q}}+1)\big),
		\end{align}
		which implies
		\begin{equation}
			\|P\|_{F^{s-1}_{p,q}}\leq C \big(\|w\|_{F^{s-1}_{p,q}}(\|v\|_{F^{s}_{p,q}}+\|u\|_{F^{s}_{p,q}})+\|\theta\|_{F^{s-2}_{p,q}}(\|\gamma\|_{F^{s-1}_{p,q}}+\|\eta\|_{F^{s-1}_{p,q}}+1)\big),
		\end{equation}
		\begin{equation}
			\|Q\|_{F^{s-2}_{p,q}}\leq C\big( \|\theta\|_{F^{s-2}_{p,q}}\|v\|_{F^{s}_{p,q}}+\|w\|_{F^{s-1}_{p,q}}(\|\gamma\|_{F^{s-1}_{p,q}}+\|\eta\|_{F^{s-1}_{p,q}}+1)\big).
		\end{equation}
		Adding these two inequalities together, we obtain 
		\begin{equation}\label{eq:P+Qestimate}
			\|P\|_{F^{s-1}_{p,q}}+\|Q\|_{F^{s-2}_{p,q}}\leq C(\|w\|_{F^{s-1}_{p,q}}+\|\theta\|_{F^{s-2}_{p,q}})A(t,x),
		\end{equation}
		where $A(t,x)\triangleq \|u\|_{F^{s}_{p,q}}+\|v\|_{F^{s}_{p,q}}+\|\gamma\|_{F^{s-1}_{p,q}}+\|\eta\|_{F^{s-1}_{p,q}}+1$.
		Applying Theorem \ref{Thm:a-priori-estimates} to the systems \eqref{eq:difference-equation}, one deduces that
		\begin{align}
			\|w\|_{F^{s-1}_{p,q}}+\|\theta\|_{F^{s-2}_{p,q}}&\leq \Big((\|w_0\|_{F^{s-1}_{p,q}}+\|\theta_0\|_{F^{s-2}_{p,q}})+C\int_{0}^te^{-C\int_{0}^\tau \|\nabla u(\tau')\|_{F^{s-1}_{p,q}}d\tau'}\nonumber\\
			&\quad\times(\|w\|_{F^{s-1}_{p,q}}+\|\theta\|_{F^{s-2}_{p,q}})A(t,x)d\tau\Big)e^{C\int_{0}^t\|\nabla u(\tau)\|_{F^{s-1}_{p,q}}d\tau}.
		\end{align}
		Hence, applying the Gronwall inequality, we reach
		\begin{equation}\label{eq:equation63}
			\|w\|_{F^{s-1}_{p,q}}+\|\theta\|_{F^{s-2}_{p,q}}\leq(\|w_0\|_{F^{s-1}_{p,q}}+\|\theta_0\|_{F^{s-2}_{p,q}})e^{C\int_{0}^t A(\tau,x)d\tau}. 
		\end{equation}
		If $(w_0,\theta_0)=(0,0)$, i.e., $v=u$ and $\eta=\gamma$, then  \eqref{eq:equation63} implies the uniqueness of the solution.\\
		
		\textbf{Continuity (continuous dependence of the solution map):} 
		Note that \eqref{eq:equation63} combined with an obvious interpolation ensures the continuity with respect to the initial data in $E^{s'}_{p,q}(T)\times E^{s'-1}_{p,q}(T))$ for any $s'<s$.
		
		In the case of $q<\infty$, from \eqref{def:S_j}, let $(u^N,\gamma^N)$ be the corresponding solution with initial data $(S_{N+1}u_0,S_{N+1}\gamma_0)$. Set $(w^N,\theta^N)=(u-u^N,\gamma-\gamma^N)$. It is not hard to see that $(w^N,\theta^N)$ solves the following system:
		\begin{equation}
			\begin{cases*}
				\partial_tw^N+(u\cdot\nabla)w^N=-(w^N\cdot\nabla)u^N+F_1(u,\gamma)-F_1(u^N,\gamma^N),\\
				\partial_t\theta^N+(u\cdot\nabla)\theta^N=-(w^N\cdot\nabla)\gamma^N+F_2(u,\gamma)-F_2(u^N,\gamma^N),\\
				w^N(0)=u_0-S_{N+1}u_0,\\
				\theta^N(0)=\gamma_0-S_{N+1}\gamma_0.
			\end{cases*}
		\end{equation}
		In a manner analogous to the proof of \eqref{eq:equation63}, we have
		\begin{equation}\label{eq:equation64}
			\|w^N\|_{E^s_{p,q}(T)}+\|\theta^N\|_{E^{s-1}_{p,q}(T)}\leq C(\|u_0-S_{N+1}u_0\|_{E^s_{p,q}(T)}+\|\gamma_0-S_{N+1}\gamma_0\|_{E^{s-1}_{p,q}(T)}).
		\end{equation}
		
		Now, we show the continuity of the
		solution map in $E^s_{p,q}(T)\times E^{s-1}_{p,q}(T)$ as $1\leq p,q<\infty$. Indeed, let $u_0,v_0,\gamma_0,\eta_0\in D(R)$. By virtue of \eqref{eq:equation63}, \eqref{eq:equation64} and  Lemma \ref{lem:Triebel-Lizorkin-properties} $(v)$, we infer
		\begin{align*}
			&\quad\|u-v\|_{L^\infty_T(F^{s}_{p,q})}+\|\gamma-\eta\|_{L^\infty_T(F^{s-1}_{p,q})}\\
			&\leq \|u-u^N\|_{L^\infty_T(F^{s}_{p,q})}+\|\gamma-\gamma^N\|_{L^\infty_T(F^{s-1}_{p,q})}+\|v-v^N\|_{L^\infty_T(F^{s}_{p,q})}+\|\eta-\eta^N\|_{L^\infty_T(F^{s-1}_{p,q})}\\
			&\quad+\|u^N-v^N\|_{L^\infty_T(F^{s}_{p,q})}+\|\gamma^N-\eta^N)\|_{L^\infty_T(F^{s-1}_{p,q})}\\
			&\leq C\big( \|u_0-S_{N+1}u_0\|_{F^{s}_{p,q}}+\|v_0-S_{N+1}v_0\|_{F^{s}_{p,q}}+\|\gamma_0-S_{N+1}\gamma_0\|_{F^{s-1}_{p,q}}+\|\eta_0-S_{N+1}\eta_0\|_{F^{s-1}_{p,q}}\big)\\
			&\quad+C\|S_{N+1}u_0-S_{N+1}v_0\|^{\frac{1}{2}}_{L^\infty_T(F^{s-1}_{p,q})}\|S_{N+1}u_0-S_{N+1}v_0\|^{\frac{1}{2}}_{L^\infty_T(F^{s+1}_{p,q})}\\
			&\quad+C\|S_{N+1}\gamma_0-S_{N+1}\eta_0\|^{\frac{1}{2}}_{L^\infty_T(F^{s-2}_{p,q})}\|S_{N+1}\gamma_0-S_{N+1}\eta_0\|^{\frac{1}{2}}_{L^\infty_T(F^{s}_{p,q})}\\
			& \leq C\big( \|u_0-S_{N+1}u_0\|_{F^{s}_{p,q}}+\|v_0-S_{N+1}v_0\|_{F^{s}_{p,q}}+\|\gamma_0-S_{N+1}\gamma_0\|_{F^{s-1}_{p,q}}+\|\eta_0-S_{N+1}\eta_0\|_{F^{s-1}_{p,q}}\big)\\
			&\quad+C2^{\frac{N}{2}}R^{\frac{1}{2}}(\|u_0-v_0\|_{F^{s-1}_{p,q}}^{\frac{1}{2}}+\|\gamma_0-\eta_0\|_{F^{s-2}_{p,q}}^{\frac{1}{2}}),
		\end{align*}
		where we used Lemma 
		\ref{lem:Snguji} in the last inequality. Since $1\leq p,q<\infty$, then for any $\varepsilon>0$, one can select $N$ to be sufficiently large, such that
		\begin{equation*}
			C\big( \|u_0-S_{N+1}u_0\|_{F^{s}_{p,q}}+\|v_0-S_{N+1}v_0\|_{F^{s}_{p,q}}+\|\gamma_0-S_{N+1}\gamma_0\|_{F^{s-1}_{p,q}}+\|\eta_0-S_{N+1}\eta_0\|_{F^{s-1}_{p,q}}\big)\leq \frac{\varepsilon}{2}.
		\end{equation*}
		Then we choose $\sigma$ small enough such that $\|u_0-v_0\|_{F^{s}_{p,q}},\|\gamma_0-\eta_0\|_{F^{s-1}_{p,q}}<\sigma$ and $C2^{\frac{N}{2}}R^{\frac{1}{2}}\sigma^{\frac{1}{2}}< \frac{\varepsilon}{4}$. Hence, we have
		\begin{equation*}
			\|u-v\|_{L^\infty_T(F^{s}_{p,q})}+\|\gamma-\eta\|_{L^\infty_T(F^{s-1}_{p,q})}\leq \varepsilon.
		\end{equation*}
		This yields the continuous dependence of the solution map. Therefore, we complete the proof of Theorem \ref{thm:well-poseness}.
	\end{proof}

	\section{Blow-up criterion for two-component Euler-poincar\'{e} system}
	In this section, we will prove the blow-up criterion of the strong solutions to system \eqref{transport-system} by means of the Littlewood-Paley decomposition and the classical energy method. In order to prove our blow-up criterion Theorem \ref{thm:Blow-up-criterion-2}, we need the following lemma.
	
	\begin{lemma}\label{thm:blow-up}
		Suppose that $d\in\mathbb{N}^+$, and $(p,q)\in[1,\infty)\times[1,\infty]$ or $p=q=\infty$. Let $(u_0,\gamma_0)\in F^s_{p,q}(\mathbb{R}^d)\times F^{s-1}_{p,q}(\mathbb{R}^d)$ with $s>\max(1+\frac{d}{p},\frac{3}{2})$ (or $s=1+d$ with $p=1$ and $d\geq 2$) and $(u,\gamma)$ be the corresponding solution to system \eqref{transport-system}. If the solution $(u,\gamma)$ blows up in finite time (i.e. the lifespan of solution $T^\star<\infty$), then
		\begin{equation*}
			\int_{0}^{T^\star}\big(\|u(\tau)\|_{L^\infty}+\|\nabla u(\tau)\|_{L^\infty}+\|\gamma(\tau)\|_{L^\infty}\big)d\tau=\infty.
		\end{equation*}
	\end{lemma}
	\begin{remark}\label{remark:blow-up}
		The maximal existence time $T^\star$ in Theorem \ref{thm:well-poseness} can be chosen independent of the regularity index $s$. Indeed, let $(u_0,\gamma_0) \in F_{p,q}^s(\mathbb{R}^d)\times F_{p,q}^{s-1}(\mathbb{R}^d)$ with $s > \max\bigl(1 + \frac{d}{p}, \frac{3}{2}\bigr)$ and some $s' \in \bigl(\max(1 + \frac{d}{p}, \frac{3}{2}), s\bigr)$. Then Theorem \ref{thm:well-poseness} ensures that there exists a unique $F_{p,q}^s(\mathbb{R}^d)\times F_{p,q}^{s-1}(\mathbb{R}^d)$ (resp., $F_{p,q}^{s'}(\mathbb{R}^d)\times F_{p,q}^{s'-1}(\mathbb{R}^d)$) solution $(u_s,\gamma_s)$ (resp., $(u_{s'},\gamma_{s'})$) to system \eqref{transport-system} with the maximal existence time $T^\star_s$ (resp., $T^\star_{s'}$). Since $F_{p,q}^s(\mathbb{R}^d) \hookrightarrow F_{p,q}^{s'}(\mathbb{R}^d)$, it then follows from the uniqueness that $T^\star_s \leq T^\star_{s'}$ and $(u_s,\gamma_s) \equiv (u_{s'},\gamma_{s'})$ on $[0, T^\star_s)$. On the other hand, if we suppose that $T^\star_s < T^\star_{s'}$, then $(u_s,\gamma_s) \in C([0, T^\star_s]; F_{p,q}^{s'}(\mathbb{R}^d)\times F_{p,q}^{s'-1}(\mathbb{R}^d))$. Hence $(u_s,\gamma_s) \in L^1(0, T^\star_s; \operatorname{Lip}(\mathbb{R}^d)\times L^\infty(\mathbb{R}^d))$, which leads to a contradiction to Lemma \ref{thm:blow-up}. Therefore, $T^\star_s = T^\star_{s'}$. 
	\end{remark}
		\begin{proof}[Proof of Lemma \ref{thm:blow-up}]
				Applying the frequency localization operator $\Delta_j$ to both sides of the first equation in
				system \eqref{transport-system}, we can find 
			\begin{equation}
				\begin{cases}
					(\partial_t+u\cdot\nabla)\Delta_ju=\Delta_j F_1+[u,\Delta_j]\cdot\nabla u,\\
					\Delta_ju|_{t=0}=\Delta_ju_0.
				\end{cases}
			\end{equation}	
					
			Let us introduce particle trajectory mapping $X(t,\alpha)$, by definition, the solution of the following ordinary differential equation
			\begin{equation}
				\begin{cases}
					\partial_tX(t,\alpha)=u(t,X(t,\alpha)),\\
					X(0,\alpha)=\alpha.
				\end{cases}
			\end{equation}		
			This implies 
			\begin{equation}
				\partial_t\big(\Delta_ju(t,X(t,\alpha))\big)=\Delta_jF_1(t,X(t,\alpha))+[u,\Delta_j]\cdot\nabla u(t,X(t,\alpha)).
			\end{equation}
			We define the Jacobian of $X(t,\alpha)$ by $J(t,\alpha)=\det\big(\nabla X(t,\alpha)\big)$. Thus, one has
			\begin{align}
				&\mathrel{\phantom{=}}\partial_t\big(J(t,\alpha)\Delta_ju(t,X(t,\alpha))\big)\nonumber\\
				&=J(t,\alpha)\textup{div}u(t,X(t,\alpha))\Delta_ju(t,X(t,\alpha))+J(t,\alpha)\Delta_jF_1(t,X(t,\alpha))\nonumber\\
				&\mathrel{\phantom{=}}+J(t,\alpha)[u,\Delta_j]\cdot\nabla u(t,X(t,\alpha)),
			\end{align}
			which yields
			\begin{align}
				&\mathrel{\phantom{=}}\big|J(t,\alpha)\Delta_ju(t,X(t,\alpha))\big|\nonumber\\
				&=|\Delta_ju_0|+\int_{0}^{t}\big|J(\tau,\alpha)\textup{div}u(\tau,X(\tau,\alpha))\Delta_ju(\tau,X(\tau,\alpha))\big|d\tau\nonumber\\
				&\mathrel{\phantom{=}}+\int_{0}^{t}\big|J(\alpha,\tau)\Delta_jF_1(\tau,X(\tau,\alpha))\big|d\tau+\int_{0}^{t}\big|J(\tau,\alpha)[u,\Delta_j]\cdot\nabla u(\tau,X(\tau,\alpha))\big|d\tau.\label{eq:equation1}
			\end{align}
			Multiplying $2^{js}$ and taking $l^q$ norm on both sides of \eqref{eq:equation1}, we get by using the Minkowski inequality that
			\begin{align}
				&\mathrel{\phantom{=}}J(t,\alpha)\left(\sum_{j\in\mathbb{Z} }\left|2^{js}\Delta_ju(t,X(t,\alpha))\right|^{q}\right)^{\frac{1}{q}}\nonumber\\
				&\leq\left(\sum_{j\in\mathbb{Z} }\left|2^{js}\Delta_ju_0\right|^{q}\right)^{\frac{1}{q}}+\int_{0}^{t}J(\tau,\alpha)\textup{div}u(\tau,X(\tau,\alpha))\left(\sum_{j\in\mathbb{Z} }\left|2^{js}\Delta_ju(\tau,X(\tau,\alpha))\right|^{q}\right)^{\frac{1}{q}}d\tau\nonumber\\
				&\mathrel{\phantom{=}}+\int_{0}^{t}J(\tau,\alpha)\left(\sum_{j\in\mathbb{Z} }\left|2^{js}\Delta_jF_1(\tau,X(\tau,\alpha))\right|^{q}\right)^{\frac{1}{q}}d\tau\nonumber\\
				&\mathrel{\phantom{=}}+\int_{0}^{t}J(\tau,\alpha)\left(\sum_{j\in\mathbb{Z} }\left|2^{js}\left([u,\Delta_j]\cdot\nabla u(\tau,X(\tau,\alpha))\right)\right|^{q}\right)^{\frac{1}{q}}d\tau.\label{eq:equation2}
			\end{align}
			Next,taking the $L^p$ norm with respect to $\alpha\in\mathbb{R}^d$ on both sides of \eqref{eq:equation2}, we get by using the Minkowski inequality that
			\begin{align}
				&\mathrel{\phantom{=}}\bigg(\int_{\mathbb{R}^d}\bigg|J(\tau,\alpha)\big(\sum_{j\in\mathbb{Z} }\big|2^{js}\Delta_ju(t,X(t,\alpha))\big|^{q}\big)^{\frac{1}{q}}\bigg|^pd\alpha\bigg)^\frac{1}{p}\nonumber\\
				&\leq\int_{0}^{t}\bigg(\int_{\mathbb{R}^d}\bigg|J(\tau,\alpha)\textup{div}u(\tau,X(\tau,\alpha))\big(\sum_{j\in\mathbb{Z} }\big|2^{js}\Delta_ju(\tau,X(\tau,\alpha))\big|^{q}\big)^{\frac{1}{q}}\bigg|^pd\alpha\bigg)^\frac{1}{p}d\tau\nonumber\\
				&\quad+\|u_0\|_{F^s_{p,q}}+\int_{0}^{t}\bigg(\int_{\mathbb{R}^d}\bigg|J(\tau,\alpha)\big(\sum_{j\in\mathbb{Z} }\big|2^{js}\Delta_jF_1(\tau,X(\tau,\alpha))\big|^{q}\big)^{\frac{1}{q}}\bigg|^pd\alpha\bigg)^\frac{1}{p}d\tau\nonumber\\
				&\quad+\int_{0}^{t}\bigg(\int_{\mathbb{R}^d}\bigg|J(\tau,\alpha)\big(\sum_{j\in\mathbb{Z} }\big|2^{js}\left([u,\Delta_j]\cdot\nabla u(\tau,X(\tau,\alpha))\big)\big|^{q}\right)^{\frac{1}{q}}\bigg|^pd\alpha\bigg)^\frac{1}{p}d\tau\nonumber,
			\end{align}
			which gives
			\begin{align}
				\|u\|_{F^s_{p,q}}
				&\leq\|u_0\|_{F^s_{p,q}}+\int_{0}^{t}\|\textup{div}u\|_{L^\infty}||u||_{F^s_{p,q}}d\tau+\int_{0}^{t}\|F_1(\tau,x)\|_{F^s_{p,q}}d\tau\nonumber\\
				&+\int_{0}^{t}\Big\|\big\|2^{js}[u,\Delta_j]\cdot\nabla u\big\|_{l^q}\Big\|_{L^p}d\tau.\label{eq:equation3}
			\end{align}
			Similar to the proof of \eqref{eq:equation3}, one has
			\begin{align}
				\|\gamma\|_{F^{s-1}_{p,q}}
				&\leq\|\gamma_0\|_{F^{s-1}_{p,q}}+\int_{0}^{t}\|\textup{div}u\|_{L^\infty}||\gamma||_{F^{s-1}_{p,q}}d\tau+\int_{0}^{t}\|F_2(\tau,x)\|_{F^{s-1}_{p,q}}d\tau\nonumber\\
				&+\int_{0}^{t}\Big\|\big\|2^{j(s-1)}[u,\Delta_j]\cdot\nabla \gamma\big\|_{l^q}\Big\|_{L^p}d\tau.\label{eq:equation4}
			\end{align}
			By means of Theorem \ref{prop:commutator-estimates}, the last terms in \eqref{eq:equation3} and \eqref{eq:equation4} are bounded by
			\begin{equation*}
				\int_{0}^{t}\|\nabla u\|_{L^\infty}\|u\|_{F^s_{p,q}}d\tau,
			\end{equation*}
			and
			\begin{equation*}
				\int_{0}^{t}\|\nabla u\|_{L^\infty}\|\gamma\|_{F^{s-1}_{p,q}}+\|\gamma\|_{L^\infty}\|u\|_{F^{s}_{p,q}}d\tau,
			\end{equation*}
			respectively. Applying Lemma \ref{lem:Triebel-Lizorkin-properties} $(iii)$ and Lemma \ref{lem:Triebel-Lizorkin-properties} $(vi)$, one can get
			\begin{align}
				\|F_1\|_{F^s_{p,q}}&\leq C(\|\nabla u\|_{L^\infty}\|u\|_{F^s_{p,q}}+\| u\|_{L^\infty}\|u\|_{F^s_{p,q}}+\|\gamma\|_{L^\infty}\|\gamma\|_{F^{s-1}_{p,q}}+\|\gamma\|_{F^{s-1}_{p,q}})\nonumber\\
				&\leq C\big((\|\nabla u\|_{L^\infty}+\|u\|_{L^\infty})\|u\|_{F^s_{p,q}}+(\|\gamma\|_{L^\infty}+1)\|\gamma\|_{F^{s-1}_{p,q}}\big),
			\end{align}
			and
			\begin{align}
				\|F_2\|_{F^{s-1}_{p,q}}&\leq \|\gamma \text{div}u\|_{F^{s-1}_{p,q}}+\|\text{div} u\|_{F^{s-1}_{p,q}}\nonumber\\
				&\leq C(\|\gamma\|_{L^\infty}\|u\|_{F^s_{p,q}}+\|\gamma\|_{F^{s-1}_{p,q}}\|\nabla u\|_{L^\infty}+\| u\|_{F^{s}_{p,q}}).
			\end{align}
			From the above inequalities, we obtain
			\begin{align}
				\|u\|_{F^s_{p,q}}+\|\gamma\|_{F^{s-1}_{p,q}}
				&\leq\|u_0\|_{F^s_{p,q}}+ \|\gamma_0\|_{F^{s-1}_{p,q}}+C\int_{0}^{t}(\|u\|_{F^s_{p,q}}+\|\gamma\|_{F^{s-1}_{p,q}})\\
				&\quad\times(\|u\|_{L^\infty}+\|\nabla u\|_{L^\infty}+\|\gamma\|_{L^\infty}+1)d\tau,
			\end{align}
			 which together with the Gronwall inequality yields
			\begin{equation}\label{eq:equation6}
				\|u\|_{F^s_{p,q}}+\|\gamma\|_{F^{s-1}_{p,q}}\leq (\|u_0\|_{F^s_{p,q}}+\|\gamma_0\|_{F^{s-1}_{p,q}})e^{C\int_{0}^{t}(\|u\|_{L^\infty}+\|\nabla u\|_{L^\infty}+\|\gamma\|_{L^\infty}+1)d\tau}.
			\end{equation}	
			Therefore, if the maximal existence time $T^\star<\infty$ satisfies \[ \int_{0}^{T^\star}\big(\|u(\tau)\|_{L^\infty}+\|\nabla u(\tau)\|_{L^\infty}+\|\gamma(\tau)\|_{L^\infty}\big)d\tau<\infty, \]
			then \eqref{eq:equation6} implies that
			\[ \limsup_{t\to T^\star}(\|u\|_{F^s_{p,q}}+\|\gamma\|_{F^{s-1}_{p,q}})<\infty, \]
			which contradicts the assumption on the maximal existence time $T^\star<\infty$. This completes the proof of Lemma \ref{thm:blow-up}.
	\end{proof}

	\begin{proof}[Proof of theorem \ref{thm:Blow-up-criterion-2}]
		Taking the $L^2(\mathbb{R}^d)$ inner product of the first and second equations in system \eqref{transport-system} with sgn$(u)|u|^{p-1}$ and sgn$(\gamma)|\gamma|^{p-1}$, respectively, then using integrating by parts and the H\"{o}lder inequality, one deduces
		\begin{equation}\label{eq:equation5}
			\frac{d}{dt}\|u\|_{L^p}\leq \|F_1(t,x)\|_{L^p}+\frac{1}{p}\|\nabla u\|_{L^\infty}\|u\|_{L^p},
		\end{equation}
		and
		\begin{equation}\label{eq:equation7}
			\frac{d}{dt}\|\gamma\|_{L^p}\leq \|F_2(\tau,x)\|_{L^p}+\frac{1}{p}\|\nabla u\|_{L^\infty}\|\gamma\|_{L^p}.
		\end{equation}
		Note that if $g=(I-\Delta)^{-1}f=G*f$ with $G(x)$ be the corresponding Green function, then
		\begin{equation}\label{eq:Green-function}
			\|D^kg\|_{L^p}\leq C\|f\|_{L^p},
		\end{equation}
		holds for all $1 < p <  \infty,k=0,1,2$, and the constant $C$ is independent of $p$. Thus, there exists a constant $C$ > 0 independent of $p$ such that
		\begin{align}\label{eq:equation71}
			\|F_1\|_{L^p}\leq C\big(\|\nabla u\|_{L^\infty}(\|u\|_{L^p}+\|\nabla u\|_{L^p})+\|\gamma\|_{L^p}(\|\gamma\|_{L^\infty}+1)\big).
		\end{align}			
		Furthermore, one readily finds that
		\begin{equation}\label{eq:equation73}
			\|F_2\|_{L^p}\leq C\|\nabla u \|_{L^\infty}(\|\gamma\|_{L^p}+1).
		\end{equation}
		
		Applying $\nabla$ to both sides of the first equation in system \eqref{transport-system} gives
		\begin{equation}\label{applying-nabla-transport-equation}
			\partial_t(\nabla u)+\nabla(u\cdot\nabla u)=\nabla F_1(u,\gamma).
		\end{equation}
		Taking the $L^2(\mathbb{R}^d)$ inner product of the above equation with sgn$(\nabla u)|\nabla u|^{p-1}$,then using integrating by parts and the H\"{o}lder inequality, we infer
		\begin{equation}\label{eq:equation74}
			\frac{d}{dt}\|\nabla u\|_{L^p}\leq\frac{2}{p} \|\nabla u\|_{L^\infty}\|\nabla u\|_{L^p}+\|\nabla F_1\|_{L^p}.
		\end{equation}
		Similar to \eqref{eq:equation71}, there exists a constant $C>0$ independent of $p$ such that
		\begin{align}\label{eq:equation72}
			\|\nabla F_1\|_{L^p}\leq C\big(\|\nabla u\|_{L^\infty}(\|u\|_{L^p}+\|\nabla u\|_{L^p})+\|\gamma\|_{L^p}(\| \gamma\|_{L^\infty}+1)\big).
		\end{align}	
		Combining the above estimate with \eqref{eq:equation5} and \eqref{eq:equation7}, and then applying the Gronwall inequality, we find that
		\begin{align*}
			\|u\|_{L^p}+\|\nabla u\|_{L^p}\leq e^{C\int_{0}^t\|\nabla u\|_{L^\infty}d\tau}\big(\|u_0\|_{L^p}+\|\nabla u_0\|_{L^p}+\int_{0}^t\|\gamma\|_{L^p}(\|\gamma\|_{L^\infty}+1)d\tau\big),
		\end{align*}
		and
		\begin{align*}
			\|\gamma\|_{L^p}\leq e^{C\int_{0}^t\|\nabla u\|_{L^\infty}d\tau}\big(\|\gamma_0\|_{L^p}+1\big).
		\end{align*}
		Letting $p\to\infty$ and recalling the assumption $F^s_{p,q}(\mathbb{R}^d)\hookrightarrow\text{Lip}(\mathbb{R}^d)$ and $F^{s-1}_{p,q}(\mathbb{R}^d)\hookrightarrow L^\infty(\mathbb{R}^d)$ as $s>1+\frac{d}{p}$ or $s=1+d$ with $p=1$, one gets
		\begin{equation}
			\|u\|_{L^\infty}+\|\nabla u\|_{L^\infty}\leq e^{C\int_{0}^t\|\nabla u\|_{L^\infty}d\tau}\big(\|u_0\|_{F^s_{p,q}}+\int_{0}^t\|\gamma\|_{L^\infty}(\|\gamma\|_{L^\infty}+1)d\tau\big),
		\end{equation}
		and
		\begin{equation}
			\|\gamma\|_{L^\infty}\leq e^{C\int_{0}^t\|\nabla u\|_{L^\infty}d\tau}\big(\|\gamma_0\|_{F^{s-1}_{p,q}}+1\big).
		\end{equation}
		Thus, we have
		\begin{equation}
			\begin{aligned}
				&\quad\|u\|_{L^\infty}+\|\nabla u\|_{L^\infty}+\|\gamma\|_{L^\infty}\\
				&\leq Ce^{C\int_{0}^t\|\nabla u\|_{L^\infty}d\tau}\big(\|u_0\|_{F^s_{p,q}}+\|\gamma_0\|_{F^{s-1}_{p,q}}+1+\|\gamma_0\|^2_{F^{s-1}_{p,q}}t(\|\gamma_0\|^2_{F^{s-1}_{p,q}}t+1)\big),
			\end{aligned} 	
		\end{equation}
		which together with \eqref{eq:equation6} and Lemma \ref{thm:blow-up} complete the proof of Theorem \ref{thm:Blow-up-criterion-2}. 
	\end{proof}

	\noindent{\bf Acknowledgments.} This work was partially supported by the National Natural Science Foundation of China under grant 11971188.

	\bibliographystyle{abbrv}
	\bibliography{ref}
	
\end{document}